\documentclass[3p]{elsarticle}
\usepackage[english]{babel}
\usepackage[latin1]{inputenc}
\usepackage{amssymb,amsmath,amsthm,amsfonts}
\usepackage{graphicx}
\usepackage[usenames,dvipsnames]{xcolor}
\usepackage{cancel}
\usepackage[colorlinks]{hyperref}
\usepackage{filecontents}
\usepackage[normalem]{ulem}
\usepackage{dsfont}
\usepackage[showonlyrefs]{mathtools}
\mathtoolsset{showonlyrefs=true}

\usepackage{tikz}
\newcommand{\grigetto}{Periwinkle}

\newcommand{\R}{\mathbb{R}}

\newcommand{\eps}{\varepsilon}

\newcommand{\la}{\lambda}

\newcommand{\into}{\int_{\Omega}}

\newcommand{\Om}{\Omega}

\DeclareMathOperator{\cp}{cap}

\newcommand{\Ccal}{{\mathcal{C}}}

\newcommand{\Hcal}{{\mathcal{H}}}

\def\XXint#1#2#3{{\setbox0=\hbox{$#1{#2#3}{\int}$ }
\vcenter{\hbox{$#2#3$ }}\kern-.6\wd0}}

\newcommand{\ind}[1]{\mathds{1}_{#1}}

\DeclareMathOperator{\diverg}{div}

\DeclareMathOperator{\supp}{supp}

\DeclareMathOperator{\trace}{trace}

\newtheorem{proposition}{Proposition}[section]
\newtheorem{theorem}[proposition]{Theorem}
\newtheorem{corollary}[proposition]{Corollary}
\newtheorem{lemma}[proposition]{Lemma}

\theoremstyle{definition}
\newtheorem{definition}[proposition]{Definition}
\newtheorem{remark}[proposition]{Remark}
\newcommand{\beq}{\begin{equation}}
\newcommand{\eeq}{\end{equation}}
\newcommand{\ben}{\begin{enumerate}}
\newcommand{\een}{\end{enumerate}}
\newcommand{\bit}{\begin{itemize}}
\newcommand{\eit}{\end{itemize}}
\newcommand{\dys}{\displaystyle}

\DeclareMathOperator{\sd}{M}
\DeclareMathOperator{\od}{\Lambda}

\begin{document}
\begin{frontmatter}

\title{Asymptotic spherical shapes in some spectral optimization problems\tnoteref{t1}}
\tnotetext[t1]{We warmly thank Bozhidar Velichkov for the very useful discussions.
Work partially supported by the project ERC Advanced Grant 2013 n. 339958:
``Complex Patterns for Strongly Interacting Dynamical Systems - COMPAT'', by the PRIN-2015KB9WPT Grant:
``Variational methods, with applications to problems in mathematical physics and geometry'', and by
the INdAM-GNAMPA group. }


\author{Dario Mazzoleni\fnref{fn1}}
\fntext[fn1]{Dipartimento di Matematica e Fisica ``N. Tartaglia''
Universit\`a Cattolica del Sacro Cuore, via Trieste 17,
25121 Brescia, Italy}
\ead{dariocesare.mazzoleni@unicatt.it}
\author{Benedetta Pellacci\fnref{fn2}}
\fntext[fn2]{Dipartimento di Matematica e Fisica
Universit\`a della Campania  ``Luigi Vanvitelli'',  via A.Lincoln 5,
Caserta, Italy.}
\ead{benedetta.pellacci@unicampania.it}
\author{Gianmaria Verzini\fnref{fn3}\corref{cor}}
\fntext[fn3]{Dipartimento di Matematica, Politecnico di Milano, p.za Leonardo da Vinci 32,  20133 Milano, Italy}
\cortext[cor]{Corresponding author.}
\ead{gianmaria.verzini@polimi.it}
%
%
%
%

%
%



\begin{abstract}
We study the optimization of the positive principal eigenvalue of an 
indefinite weighted problem,
associated with the Neumann Laplacian in a box $\Om\subset\R^N$, which arises in
the investigation of the survival threshold in population dynamics.
When trying to minimize such eigenvalue with respect to the weight, one is led to
consider a shape optimization problem, which is known to admit no spherical optimal
shapes (despite some previously stated conjectures).
We investigate whether spherical shapes can be recovered in some singular perturbation
limit. More precisely we show that, whenever the negative part of the weight diverges,  the above shape optimization problem approaches in the limit
the so called spectral drop problem, which involves the minimization of the first
eigenvalue of the mixed Dirichlet-Neumann Laplacian. Using $\alpha$-symmetrization 
techniques on cones, we prove that, for suitable choices
of the box $\Om$, the optimal shapes for this second problem are indeed spherical.
Moreover, for general $\Omega$, we show that small volume spectral drops are
asymptotically spherical, centered near points of $\partial\Omega$ having largest mean curvature.
%

%

\noindent\textbf{Résumé}\\

\noindent On etudie l'optimisation de la valeur propre principale positive d'un
probl\`eme avec un poids ind\'efini,
associ\'e au Laplacien avec conditions au bord de Neumann dans une bo\^ite $ \Om \subset \R^N $, qui appara\^it dans
l'\'etude naturellement comme valeur limite de survie en dynamique des populations.
En essayant de minimiser une telle valeur propre par rapport au poids, on est amen\'e \`a
consid\'erer un probl\`eme d'optimisation de forme, qui n'admet pas de formes sph\'eriques comme solutions (chose qui contredit certaines conjectures \'enonc\'ees pr\'ec\'edemment).
Nous \'etudions si des formes sph\'eriques peuvent \^etre recouvr\'ees comme solutions de certaines perturbations singuli\`eres du probl\`eme. Notamment, on d\'emontre que, si la partie n\'egative du poids diverge, le probl\`eme d'optimisation de forme se rapproche \`a la limite \`a un probl\`eme de ``goutte spectrale'', cette \`a dire un probl\`eme de minimisation de la premi\`ere
valeur propre du Laplacien avec conditions mixtes Dirichlet-et-Neumann au bord. En utilisant techniques de $\alpha$-sym\'etrisation
 sur les c\^ones, on d\'emontre que, pour des choix opportunes
de la bo\^ite $ \Om $, les formes optimales pour ce second probl\`eme sont bien sph\'eriques.
De plus, pour $ \Omega $ g\'en\'erique, on d\'emontre que les gouttes spectrales  s'approchent asymptotiquement -lorsque leur volume devient infinit\'esimal- \`a des gouttes sph\'eriques, centr\'ees près des points de $ \partial \Omega $ ayant courbure moyenne maximale.
\end{abstract}

\end{frontmatter}

\section{Introduction}\label{sec:problem}
In this paper we are concerned with two spectral shape optimization problems, both settled in a box, that is, a bounded domain (open and connected set) of $\R^N$, 
$N\geq 2$, with Lipschitz boundary, denoted by $\Om$.

The first problem we consider is an optimal design problem related to the \emph{survival threshold} in population dynamics \cite{MR1014659,ly}. Here, the cost
is the positive principal eigenvalue of the weighted Neumann Laplacian. More precisely, for a sign-changing weight $m \in
L^\infty(\Omega)$ we consider the eigenvalue problem
\begin{equation}\label{eq:eigprobOD}
\begin{cases}
 -\Delta u = \lambda m u & \text{in }\Omega\\
 \partial_\nu u = 0 & \text{on }\partial\Omega.
\end{cases}
\end{equation}
A principal eigenvalue for \eqref{eq:eigprobOD} is a number $\lambda$
having a positive eigenfunction. It is well known that, in case $m^+$ and $m^-$ are
both nontrivial, \eqref{eq:eigprobOD} admits two principal eigenvalues, $0$ and
$\lambda(m)$. Moreover, $\lambda(m)>0$ if and only if $\int_\Omega m <0$, in which case
\begin{equation}\label{eq:def_lambda(m)}
\lambda(m):= \min \left\{\dfrac{\int_\Omega |\nabla u|^2\,dx}{\int_\Omega m u^2\,dx} :  u\in H^1(\Omega),\ \int_\Omega m u^2\,dx>0\right\}.
\end{equation}
Problem \eqref{eq:eigprobOD} is  the 
stationary linearized equation associated with classical 
reaction-diffusion models for the dynamic of a population, of density $u$, 
inhabiting a heterogenous environment (see \cite{fisher, kpp, MR2191264}). 
In this context, $m(x)$
describes the intrinsic growth rate of the population at $x$ (positive in favorable
sites, negative in hostile ones), and $\lambda(m)$ is related to the survival chances
of the population: a smaller value of $\lambda(m)$ provides better chances of species
survival. For this reason, the problem of minimizing $\lambda(m)$, with $m$ varying in
some suitable class, has been widely considered in the literature: we postpone a
detailed discussion of the state of the art for such problem to Section
\ref{sec:opt_thresh} ahead, while here we just describe some results which motivate our
study.

When the mean $\int_\Omega m$ is fixed, as well as lower and upper bounds $-\underline{m}
\le m \le \overline{m}$, it is known \cite{ly} that the infimum of $\lambda(m)$ is achieved by a
bang-bang (i.e. piecewise constant) optimal weight $m^* = \overline{m}\ind{D^*} - \underline{m}\ind{\Omega\setminus D^*}$,
where the measurable set $D^*$ can be chosen to be open. For this reason one can
equivalently consider the minimization over the class of bang-bang weights $\overline{m}\ind{D} -
\underline{m}\ind{\Omega\setminus D}$, under a volume constraint on $D$ in order to fix
the average of $m$. Finally, up to a scaling, we can choose $\overline{m} = 1$ and obtain the first shape optimization problem that we consider.
\begin{definition}
Let $\beta>0$ and $0<\delta<\dfrac{\beta|\Omega|}{\beta+1}$. For any $D\subset\Omega$ such that $|D| = \delta$
we define, with some abuse of notation, the eigenvalue of the corresponding bang-bang weight as
\begin{equation}\label{eq:def_lambda_beta_D}
\lambda(\beta,D):=\lambda(\ind{D} - \beta\ind{\Omega\setminus D})= \min \left\{
\dfrac{\int_\Omega |\nabla u|^2\,dx}{\int_D u^2\,dx - \beta \int_{\Omega\setminus D} u^2\,dx} :  u\in H^1(\Omega),\ \int_D u^2\,dx>\beta \int_{\Omega\setminus D} u^2\,dx\right\},
\end{equation}
and the optimal design problem for the survival threshold as
\begin{equation}\label{eq:def_od}
\od(\beta,\delta)=\min\Big\{\lambda(\beta,D):D\subset \Om,\mbox{ measurable, }|D|=\delta\Big\}.
\end{equation}
\end{definition}
As we mentioned, any minimizer $D^*$ achieving $\od(\beta,\delta)$ is open, up  a negligible set: actually, it is
a superlevel of a corresponding eigenfunction of \eqref{eq:eigprobOD}. Since $D^*$ represents the favorable patch
of the habitat which optimizes the survival chances, natural questions arise about its shape and its location inside
$\Omega$. In the case of Dirichlet boundary conditions, Cantrell and Cosner \cite{MR1014659} pointed out that if $\Omega$
is a ball, then $D^*$ is a ball too, concentric with $\Omega$. On the other hand, in the case of Neumann boundary conditions and spatial dimension $N=1$, it is known
\cite{MR1105497,ly,llnp} that any $D^*$ is a connected interval which touches the boundary of $\Omega$. Based on these results, as well as on numerical simulations,
a commonly stated conjecture was that the ball, or the intersection of a ball with $\Omega$, achieves $\od(\beta,\delta)$, at least
for some choice of the parameters or of the box \cite{MR2214420,MR2494032,haro}. In particular, in case $\Omega$ is a
rectangle and $\delta$ is not too large, it was conjectured that $D^*$ would be a quarter of a disk centered at a vertex of
$\Omega$. Notice that the case of rectangular boxes is not only
interesting as a prototypical example, but also because its study is equivalent to that
of a periodically fragmented environment. For easier terminology, in the following we
say that a shape $D^*$ is \emph{spherical} if $D^*=\Om\cap B_{r(\delta)}(x_0)$, for a
suitable $x_0$, and $r(\delta)$ is such that $|D^*|=\delta$.

Rather surprisingly, all these conjectures about optimal spherical shapes were
recently disproved by Lamboley et al.
in \cite{llnp}: if $D^*$ is a minimizer in any $N$-dimensional rectangle, for any
choice of $\beta$ and $\delta$, then $\partial D^*$ can not contain
any portion of sphere. One ingredient of their proof is a generalization of ideas by Henrot and Oudet \cite{ho}; indeed, arguing as in \cite{ho}, one realizes that the main obstruction to the
presence of spherical shapes for $\od(\beta,\delta)$ is provided by the part of $\partial\Omega$ which lays faraway from
$D^*$. The main aim of this paper is to show that, in some singular perturbation regimes, the influence of such part
of $\partial\Omega$ becomes negligible, and thus optimal spherical shapes can be obtained in the asymptotic limit.

In order to pursue this goal, there are two possible choices: one can either consider very small favorable regions,
letting $\delta\to0$, or very hostile unfavorable ones, in case $\beta\to+\infty$. To start with, we focus on this second
possibility, taking advantage of the following result (see also~\cite{derek}).
\begin{lemma}\label{lem:intro_beta2infty}
Let $\Om\subset\R^N$ be a bounded Lipschitz domain. For any open, Lipschitz $D\subset\Omega$, $0<|D|<|\Om|$, we have
\[
\lim_{\beta\to+\infty} \lambda(\beta,D) = \mu(D,\Omega),
\]
where $\mu(D,\Om)$ is the first eigenvalue of the mixed Dirichlet-Neumann problem
\begin{equation}\label{eq:limprob}
\begin{cases}
-\Delta u = \mu(D,\Om) u &\text{in } D \\
u=0 &\text{on } \partial D\cap \Omega\\
\partial_\nu u=0 &\text{on } \partial D\cap \partial \Omega.
\end{cases}
\end{equation}
\end{lemma}
The above lemma suggests that minimizers of the optimal design problem
$\od(\beta,\delta)$ should be related, for $\beta$ large, to minimizers
of the mixed Dirichlet-Neumann eigenvalue problem, among subdomains of measure
$\delta$. This leads to the second shape optimization problem that we consider,
i.e. the \emph{spectral drop} problem,
which was introduced and studied by Buttazzo and Velichkov in \cite{buve}.
Lipschitz subdomains are not enough, in order to settle this problem, and one is led 
to consider \emph{quasi-open} subsets of $\Omega$: $D$ is quasi-open
if it is open, up to sets of arbitrarily small capacity (see Section 
\ref{sec:spectraldrop} for details about capacity and quasi-open sets,
and Section \ref{sec:beta2infty} for a generalization of Lemma \ref{lem:intro_beta2infty} to quasi-open $D$).
\begin{definition}
Let $0<\delta<|\Om|$. For any quasi-open $D\subset\Omega$ such that $|D| = \delta$ we define the mixed Dirichlet-Neumann eigenvalue as
\begin{equation}\label{eq:def_mu_D}
\mu(D,\Om):=\min{\left\{\frac{\int_{\Om}|\nabla u|^2\,dx}{\int_\Om u^2\,dx}:u\in H^1_0(D,\Om)\setminus\{0\}\right\}},
\end{equation}
where
\[
H^1_0(D,\Om):=\left\{u\in H^1(\Om):u=0\text{ q.e. on }\Om\setminus 
D\right\},
\]
and $q.e.$ stands for \emph{quasi-everywhere}, which means up to sets of zero 
capacity.  Then, the \emph{spectral drop problem} is
\begin{equation}\label{eq:def_sd}
\sd(\delta)=\min{\Big\{\mu(D,\Omega):D\subset \Om,\;\mbox{quasi-open, }|D|=\delta\Big\}}.
\end{equation}
\end{definition}
It is known by~\cite{buve} that $\sd(\delta)$ is achieved, in the class of quasi-open
sets. More informations on this mixed boundary conditions problem are detailed in
Section~\ref{sec:mixed}.

Our first main result concerns the connection between the two optimal partition problems.
\begin{theorem}\label{thm:equality}
Let $\Om\subset\R^N$ be a bounded Lipschitz domain, $0<\delta<|\Om|$, $\beta>\frac{\delta}{|\Om|-\delta}$ and $\eps\in \left(\frac{\delta}{\beta}, |\Omega|-\delta\right) $. 
Then, we have
\[
\sd(\delta+\eps)\left(1-\sqrt{\frac{\delta}{\eps\beta}}\right)^{2}\leq \od(\beta,\delta)\leq \sd(\delta)
\]
As a consequence, for every $0<\delta<|\Om|$, 
\[
\lim_{\beta\to+\infty} \od(\beta,\delta) = \sd(\delta),
\]
i.e.
\[
\lim_{\beta\to+\infty}\min_{|D|=\delta}\lambda(\beta,D) = \min_{|D|=\delta} \lim_{\beta\to+\infty}\lambda(\beta,D).
\]
\end{theorem}
The proof of such result is rather delicate, and even the easier 
inequality, i.e. $\limsup_{\beta\to+\infty} \od(\beta,\delta) \le \sd(\delta)$, 
requires non-trivial arguments: indeed, for general (quasi-)open $D$, in principle 
the eigenfunctions associated to $\lambda(\beta,D)$ converge to 0, as $\beta\to+\infty$, only a.e. outside $D$, and not also q.e..

Theorem \ref{thm:equality} immediately allows to transfer information from the spectral drop
problem to the survival threshold one. For instance an immediate consequence is the following.
\begin{corollary}\label{coro:intersectsboundary}
Let $\Omega\subset\R^N$ be a bounded Lipschitz domain and $D^*$ be an optimal set for $\sd(\delta)$ with $\mathcal H^{N-1}(\partial D^*\cap \partial \Om)>0$.
Then for all quasi-open $D\subset\Omega$ with $|D|=\delta$ and
$\Hcal^{N-1}(\partial D\cap\partial\Omega) = 0$, and all $\beta$ sufficiently large,
depending on $D$, we have
\[
\la(\beta,D)>\la(\beta,D^*).
\]
\end{corollary}
In particular, the assumptions of Corollary~\ref{coro:intersectsboundary} hold if $N\leq 4$ and $\Om\subset\R^N$ is smooth (see Remark~\ref{rem:tocca il bordo}). This suggests that optimal sets for $\od(\beta,\delta)$ should touch the boundary of 
$\Omega$.
 Though this latter property is somehow expected, for general 
$\beta$ it is only known in dimension $N=1$ (as we already mentioned) and in 
the case of rectangular domains, as a consequence of the monotonocity of the bang-bang 
optimal weight \cite[Proposition 5]{llnp}.

Once the connection between the survival threshold problem and the spectral drop one is 
established, the next question we address
is whether the latter admits spherical optimal shapes. The aforementioned ideas by 
Henrot and Oudet partially apply also to the spectral drop
problem, but in this case some space for spherical shapes is left. More precisely, we 
can show the following.
\begin{proposition}\label{prop:henrotoudet}
Let $\Om\subset\R^N$ be a bounded Lipschitz domain, and $0<\delta<|\Om|$.
Let $D^*$ denote an optimal set for $\sd(\delta)$ and assume that
$\partial D^* \cap \Omega$ contains some non empty, relatively open portion of sphere,
centered at some $x_0$.
Then any eigenfunction achieving $\mu(D^*,\Om)$ is radially symmetric in $D^*$, and
\begin{itemize}
\item any regular surface contained in $\partial D^*\cap \Omega$ is a portion of a 
sphere centered at $x_0$;
\item any regular surface contained in $\partial D^*\cap \partial\Omega$ is either a 
portion of a cone
with vertex at $x_0$, or a portion of a sphere centered at $x_0$.
\end{itemize}
\end{proposition}
\begin{figure}[ht]
\begin{center}
\hfill
\begin{tikzpicture}
\draw[draw={\grigetto}, fill={\grigetto}] (0,0) -- (2.5,0) arc (0:70:2.5) -- cycle;
\draw[draw={\grigetto}, fill={\grigetto}] (35:1.6) node {$D^*$};
\draw[thick, dashed] (2.5,0) arc (0:70:2.5);
\draw[thick] (0,0) -- (3,0) to[out=0,in=-180+80] (2.3,2) node[right] {$\Omega$} 
to[out=80,in=70] (70:2.8) -- cycle;
\end{tikzpicture}
\hfill
\begin{tikzpicture}
\draw[draw={\grigetto}, fill={\grigetto}] (70:1) arc (70:0:1) -- (2.5,0) arc (0:70:2.5) -- cycle;
\draw[draw={\grigetto}, fill={\grigetto}] (35:1.8) node {$D^*$};
\draw[thick, dashed] (2.5,0) arc (0:70:2.5);
\draw[thick] (70:1) arc (70:0:1) -- (3,0) to[out=0,in=-180+80] (2.3,2) node[right] {$\Omega$} 
to[out=80,in=70] (70:2.8) -- cycle;
\end{tikzpicture}
\hfill
\begin{tikzpicture}
\draw[draw={\grigetto}, fill={\grigetto}] (70:1) arc (70:0:1) -- (2.5,0) arc (0:70:2.5) -- 
cycle;
\draw[draw={\grigetto}, fill={\grigetto}] (35:1.8) node {$D^*$};
\draw[thick, dashed] (2.5,0) arc (0:70:2.5) (70:1) arc (70:0:1);
\draw[thick] (70:.8) to[out=-180+70,in=-180] (.7,0) -- (3,0) to[out=0,in=-180+80] (2.3,2) 
node[right] {$\Omega$} to[out=80,in=70] (70:2.8) -- cycle;
\end{tikzpicture}
\hfill
\caption{possible shapes of $D^*$, 
according to  Proposition \ref{prop:henrotoudet}. The Dirichlet boundary $\partial D^*\cap\Omega$ is dashed.
\label{fig:HO}}
\end{center}
\end{figure}
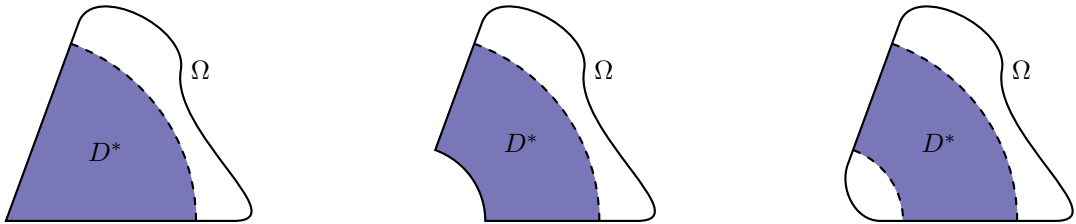
In Fig. \ref{fig:HO} we illustrate the three cases allowed by the necessary condition stated in the above proposition. In particular, in case $\Omega$ is a rectangle, the above result does not exclude spherical
spectral drops, centered at a vertex. Actually,
using symmetrization techniques borrowed from \cite{MR876139,MR963504}, we can show
that spherical shapes appear, at least when
$\delta$ is not too large. For this result we exploit relative isoperimetric inequalities obtained in \cite{cianchi,MR3335407}. We state
our results for $N$-dimensional polyhedra, even though this holds true for any convex $\Omega$ which coincides locally
with its tangent cone having smallest solid angle.
\begin{theorem}\label{thm:main_polythope}
Let $\Omega \subset \R^N$ be a bounded, convex polytope. There exists $\bar\delta>0$ such that, for any $0<\delta< \bar\delta$:
\begin{itemize}
\item $D^*$ is a minimizer of the spectral drop problem
in $\Omega$, with volume constraint $\delta$, if and only if
$D^*=B_{r(\delta)}(x_0)\cap\Omega$, where $x_0$ is a vertex of $\Omega$ with the smallest solid angle;
\item if $|D|=\delta$ and $D$ is not a spherical shape as above, then, for $\beta$ sufficiently large,
\[
\lambda(\beta,D)> \lambda (\beta,B_{r(\delta)}(x_0)\cap \Omega).
\]
\end{itemize}
In addition, in case $\Omega = (0,L_1) \times (0,L_2)$, with $L_1\le L_2$, we have $\bar\delta\ge L_1^2/\pi$, and any minimizing
spectral drop is a quarter of a disk centered at a vertex of $\Omega$ for any $0<\delta< L_1^2/\pi$.
\end{theorem}
Let us notice that, in general, $\bar\delta$ is strictly smaller than the largest 
$\delta$ such that $B_{r(\delta)}(x_0)\cap\Omega$ is a cone (see Remark \ref{rem:bardelta}). 
In~\cite{mapeve_matrix} we obtain explicit estimates on $\bar\delta$ in the case of a polygon in $\R^2$, as well as quantitative estimates on $\od(\beta,\delta)$ and 
$\sd(\delta)$. 

As a consequence of the above theorem, we have that the conjecture about circular optimal shapes in a rectangle, which is false
for the survival threshold problem, for any $\beta$, becomes true in the singular limit $\beta\to+\infty$. This somehow
helps to understand the different results obtained in \cite{haro} and \cite{llnp}.

As stated in Proposition \ref{prop:henrotoudet}, in case $\partial\Omega$ does not contain portions of spheres or cones, one can not have spherical spectral drops in $\Omega$. Motivated by this, the last question we address in this paper is whether spherical shapes can be recovered also in this case,
up to the further singular perturbation $\delta\to0$. Actually, we show that this is the case: when the volume $\delta$ becomes
very small, then $\sd(\delta)$ tends to be achieved by portions of spheres, centered at points $x_0\in\partial \Omega$ having large mean curvature $H(x_0)$.
\begin{theorem}\label{thm:main_small}
There exist explicit universal constants $0<\overline{C}_N<\underline{C}_N$ such that,
for every $\Omega$ of class $C^2$ and for any $D^*$ achieving $\sd(\delta)$ we have
\begin{equation}\label{eq:pinch}
-\underline{C}_{N}\hat H + o(1) \le
\frac{\mu(D^*, \Om)-\mu(B^+_1,\R^N_+)|B_1^+| ^{2/N}  \cdot \delta^{-2/N} }{\mu(B^+_1,\R^N_+)|B_1^+| ^{2/N} \delta^{-1/N}}
\le -\overline{C}_{N}\hat H+o(1),\qquad\text{as }\delta\to0^+,
\end{equation}
where $\hat H = \max_{x_0\in \partial \Om}H(x_0)$, the maximum of the mean curvature of $\partial\Omega$. In particular,
\[
\mu(D^*, \Om)-\mu(B_{r(\delta)}(x_0)\cap\Omega,\Omega) = o(\delta^{-2/N})
\qquad\text{as }\delta\to0^+,
\]
for every $x_0\in\partial\Omega$, where $r(\delta)$ is such that $|B_{r(\delta)}(x_0)\cap\Omega|=\delta$.
\end{theorem}
The above theorem implies that the exact first order term in the expansion of 
$\mu(D^*, \Om)$ is given by the eigenvalue of a portion of sphere centered at any  
$x_0\in\partial\Omega$, and \eqref{eq:pinch} yields $\mu(D^*, \Om) \sim C 
\delta^{-2/N}$ as $\delta\to0^+$. 

In addition, we have a bound on the second order term, depending on the maximal mean 
curvature. More precisely the estimate from above is inspired by computations performed 
in \cite{yyli}. On the other hand, the estimate from below exploits sharp relative 
isoperimetric inequalities proved by Fall in \cite{fall}. As explained in such paper, asymptotic spherical optimal shapes for isoperimetric inequalities with small volume
constraints have been object of large interest in the last years. We refer 
to \cite{MR1897460} for absolute isoperimetric inequalities in manifold without boundary, and to \cite{fall} and references therein for relative isoperimetric inequalities in manifold with boundary.
From Theorem~\ref{thm:main_small} one can also deduce a quantitative estimate of how much spherical shapes are far from being optimal for \eqref{eq:def_od}, under the regime of $\delta$ small and $\beta$ large.
\begin{corollary}\label{cor:post_small}
Under the notation of Theorem \ref{thm:main_small} we have
\begin{equation}\label{eq:pinchBr}
1\leq \frac{\la(\beta,B_{r(\delta)}(x_0))}{\od(\beta,\delta)}\le
\left(1+A \delta^{1/N} +o(\delta^{1/N})\right)
\left(1+B \beta^{-1/3} +o(\beta^{-1/3})\right),\qquad
\text{as }\delta\rightarrow 0^+,\ \beta\to+\infty,
\end{equation}
where $A = \underline{C}_N\hat H - \overline{C}_N H(x_0)$ and $B = \frac2N + 2$.
\end{corollary}
It is natural to ask whether also $\lim_{\delta\to0}
\od(\beta,\delta)$, for fixed $\beta>0$, is achieved by
asymptotic spherical shapes. This appears to be a difficult question and is under
study. In particular, symmetrization techniques can not be applied to this problem in a
direct way, since the eigenfunctions related to this problem are positive in the whole $\Omega$, and one should symmetrize also superlevel sets having measure near $|\Omega|$.

As a final remark, we observe that the main theme of this paper consists in using Theorem \ref{thm:equality} in order to deduce properties of $\od(\beta,\delta)$, for $\beta$ large, from properties of $\sd(\delta)$.
However, also the other direction of such relation can be exploited. In particular,
since numerical simulations for $\od(\beta,\delta)$ are easier to implement, these may
be used to deduce numerical properties for $\sd(\delta)$.

\textbf{Plan of the paper.}
The paper is structured as follows: Section \ref{sec:back} is devoted to the discussion 
of the state of the art about the optimal survival threshold problem and the spectral 
drop problem, and to the proof of Proposition~\ref{prop:henrotoudet}. In Section~\ref{sec:beta2infty} we study the 
asymptotic of problem~\eqref{eq:def_od} as $\beta\rightarrow \infty$ and prove 
Theorem~\ref{thm:equality}  and Corollary \ref{coro:intersectsboundary} (beyond a suitable generalization of Lemma 
\ref{lem:intro_beta2infty}). In Section~\ref{sec:asymptspecrtaldrop} we first study the link between $\alpha$-symmetrization, relative isoperimetric inequalities and the principal eigenvalue of the mixed Dirichlet-Neumann problem: this allows us to prove Theorem~\ref{thm:main_polythope}. Then, we study the asymptotics as $\delta \rightarrow 0$ for problem~\eqref{eq:def_sd}, and complete the proof of Theorem~\ref{thm:main_small} and Corollary \ref{cor:post_small}.

\textbf{Notation.}
In this paper we will use the following notation.
\begin{itemize}
\item
\(
\omega_N = |B_1|, \; N\omega_N=\Hcal^{N-1}(\partial B_1)
\)
where, as usual, $|\cdot|$ denotes the $N$-dimensional Lebesgue measure,
$\Hcal^{N-1}(\cdot)$ denotes the $(N-1)$-dimensional Hausdorff measure, and
$B_R\subset\R^N$ is the ball of radius $R$. 
\item
$B_R^+=B_R\cap
\R^N_+=B_R\cap\{x_N>0\}$, where $\R^N\ni x = (x',x_N)\in \R^{N-1}\times \R$.
\item
$\ind{D}$ is the piecewise constant function such that
$\ind{D}(x)=1$ if $x\in D$ and $\ind{D}(x)=0$ elsewhere.
\end{itemize}

\section{Preliminaries and background}\label{sec:back}

\subsection{The optimal survival threshold problem}\label{sec:opt_thresh}

Our main motivation for studying the optimal design problem $\od(\beta,\delta)$ 
comes from its connection with the optimal spatial arrangement of
favourable and unfavourable regions for a species to survive.

A classical model for spatial dispersal of a population in a
heterogenous environment is the reaction-diffusion equation
of logistic type introduced by Fisher \cite{fisher} and
Kolmogorov, Petrovsky and Piskunov \cite{kpp}
(see also \cite{MR2191264})
\begin{equation}\label{eq:readiff}
\begin{cases}
 u_t - d\Delta u = m u - u^2 & x\in\Omega,\ t>0\\
 u(x,0) = u_0(x) \ge 0 & x\in\Omega\\
 \partial_\nu u = 0 & x\in\partial\Omega,\ t>0,
\end{cases}
\end{equation}
where $u=u(x,t)\ge0$ is the density of the population in the spatial
region $\Omega$, the Neumann boundary conditions describe
the fact that there is no flux at $\partial \Omega$, $d>0$ is the motility
coefficient of the species, and $m=m(x)$ denotes the intrinsic
growth rate of the population. As explained in \cite{MR2214420}
a widely studied approximation of a heterogeneous habitat
is a  patchwork of differentiated  environments each with a defined
structure, this is the so called ``patch model'', where it is assumed that
the intrinsic growth rate $m$ varies with patches, so that we can
distinguish the favourable zones  $\{ x: m(x)>0\}$
and the hostile ones $\{ x: m(x)<0\}$.
Concerning solutions of \eqref{eq:readiff}, two alternatives may
occur as $t\to +\infty$: either the population undergoes extinction, i.e.
$u(x,t)\to 0$, or it survives, i.e. $u(x,t)$ converges to a nontrivial stationary
solution. Actually, the survival for every nontrivial initial datum is equivalent to
the existence of a nontrivial stationary solution, which in turn is equivalent, as first shown by Skellam in \cite{MR0043440}
(see also \cite{MR1014659, MR2191264}),
to the survival condition
\[
d \lambda(m) < 1,
\]
where $\lambda(m)$ is defined in \eqref{eq:def_lambda(m)}. This condition is
particularly significant when $\lambda(m)>0$, or equivalently when $\int_\Om m<0$ and $m>0$ in a set of positive measure.
In this situation, $\lambda(m)$ acts as a survival threshold for the population,
and its minimization increases the chances of survival. This provides
the determination of the optimal spatial arrangement of the favorable and unfavorable
patches of the environment for the species to survive, and it is important for 
the conservation of species with limited resources $\int_\Om m$.

Following \cite{MR1014659, ly}, we are led
to consider the following optimization problem:
\[
c(\beta,m_0)=\min_{{\mathcal M}(\beta,m_0)}\lambda(m),
\]
where, for positive $0<m_0<\beta$, the (non-empty) set ${\mathcal M}(\beta,m_0)$ is defined as
\[
\mathcal{M}(\beta,m_0):=\left\{m\in L^{\infty}(\Omega) : -\beta \leq  m \leq 1,\  {m^{+}\not\equiv 0},\ \into m(x)dx\leq -m_0\right\},
\]
{where $m^{+}$ stands for the positive part of the function $m$}.
Notice that the general case $-\underline{m}\le m\le \overline{m}$,
$\into m \le m_0'<0$ can always be reduced to the above one, by factorizing
$\overline{m}$. 
\begin{theorem}[\cite{ly}]\label{thm:base}
If $0<m_0<\beta$ then $c(\beta,m_0)$ is achieved. For any minimizer $m^*$ there exists
$D^*\subset \Omega$ such that
\[
m^* = \ind{D^*}-\beta\ind{\Omega\setminus D^*} = (1+\beta)\ind{D^*} - \beta.
\]
Moreover $D^*$ is an open set, up to zero measure sets: indeed, if $u>0$
is an eigenfunction associated with $c(\beta,m_0)=\la(m^*)$, then $u$ is $C^{1,\alpha}$ for all $\alpha\in(0,1)$,
any level set of $u$ has zero Lebesgue measure, and we can choose
\[
D^{*}=\{x\in \Omega : u(x)>t\}
\]
(for some $t>0$). Furthermore,
\[
(1+\beta)|D^*| - \beta |\Omega| = -m_0 \qquad\text{(i.e. $|D^*| = \dfrac{\beta |\Omega| - m_0}{\beta + 1}$)}.
\]
\end{theorem}
This suggests to define, for $\beta>0$ and $0<\delta<\dfrac{\beta|\Om|}{\beta+1}$, the
class of weights
\[
\mathcal{N}(\beta,\delta):=\left\{m\in L^{\infty}(\Omega) :
m = (1+\beta)\ind{D} - \beta|\Omega|,\ D\subset\Omega,\ |D|=\delta\right\},
\]
so that, recalling \eqref{eq:def_od},
\[
\od(\beta,\delta)=\min_{{\mathcal N}(\beta,\delta)}\lambda(m),
\]
and the minimization can be equivalently performed among measurable, quasi-open or open sets $D$.
As we have prescribed $\delta$ in the definition  of $\mathcal{N}$,
$\od(\beta,\delta)$ is achieved by
a potential $m$ satisfying $\into m=(1+\beta)\delta-\beta|\Om|$, so that
the following  consequence of Theorem \ref{thm:base} holds.
\begin{corollary}
$\mathcal{N}(\beta,\delta)\subset\mathcal{M}(\beta,\beta(|\Om|-\delta) - \delta)$ and
\[
c(\beta,\beta(|\Om|-\delta) - \delta) = \od \left(\beta,\delta\right).
\]
\end{corollary}
\begin{remark}\label{rem:monot}
The constraint on the measure of the set $D$ in the  optimization problem above can be equivalently taken as $|D|\leq \delta$,
thanks to the monotonicity of the eigenvalue with respect to the weight $m$ (see~\cite[Lemma~2.3]{ly}):
\begin{equation}\label{eq:weight}
m_1\leq m_2 \qquad \Longrightarrow \qquad \la(m_1)\geq \la(m_2).
\end{equation}
Indeed, introducing the  problem
\[
\tilde{d}(\beta,\delta)=\min_{\widetilde {\mathcal N}(\beta,\delta)}\lambda((1+\beta)\ind{D} - \beta),\quad \widetilde {\mathcal N}=\left\{m\in L^{\infty}(\Omega) :
m = (1+\beta)\ind{D} - \beta,\ D\subset\Omega,\ 0<|D|\le\delta\right\},
\]
it is obvious that $\tilde{d}(\beta,\delta)\leq  \od(\beta,\delta)$.
On the other hand,  if $\widetilde D$ is an optimal set
with $|\widetilde D|<\delta$, as $\delta<|\Omega|$ there exists
a set $E\subset \Om\setminus \widetilde D$ with $|E|=\delta-|
\widetilde D|$. Then $|\widetilde{D}\cup E |=\delta$ and
\eqref{eq:weight} yields
\[
\la_{1}((1+\beta)\ind{\widetilde{D}\cup E}-\beta)\leq \la_{1}((1+\beta)\ind{\widetilde{D }}-\beta),
\]
showing that $\tilde{d}(\beta,\delta)= \od(\beta,\delta)$.
\end{remark}
\begin{remark}
Let us observe that a closely related approach to the study of 
\eqref{eq:readiff} leads to define the principal eigenvalue
\[
\gamma(m):= \min \left\{
\dfrac{\int_\Omega |\nabla u|^2\,dx- \into mu^{2}\,dx}{\int_\Omega u^2\,dx} :  
u\in H^1(\Omega),\ u\not\equiv0\right\},
\]
and to conclude that the species $u(x,t)$ survives if and only if $\gamma(m)<0$
(see \cite{MR2214420,haro}). As shown in
\cite[Theorem 13]{MR1796024} (see also \cite[Section 2.2]{llnp}), it turns out that $\gamma(m)$ is also
minimized by a bang-bang weight; in addition, it is possible
to pass from a minimizer  of $\gamma$ to a minimizer of 
$\lambda$ via a change of the coefficients in the definition of the
weight, so that our results also apply in this related context. 
\end{remark}

Finally, let us mention that the optimization of $\lambda(m)$ has been investigated
also in different, although related, settings: with pointwise constraints
for positive weights with Dirichlet boundary conditions, see \cite[Chapter 9]{henrot}
and references therein; in the framework of composite membranes
\cite{MR572958,MR2473259,MR2421158}; in the  case of the 
$p$-Laplace operator in
\cite{dego}; when analyzing best dispersal strategies in spatially heterogeneous
environments, where also non-local diffusion is allowed \cite{MoPeVe,MR3771424}.

\subsection{The spectral drop problem}\label{sec:mixed}\label{sec:spectraldrop}
First of all we recall some notions that are useful when dealing with optimization problems involving the space $H^1(\Om)$. A more detailed presentation can be found for example in~\cite{hp,bubu} and in~\cite{buve}, for the parts peculiar to the mixed Dirichlet-Neumann setting.
\begin{definition}
Let $E\subset \R^N$ be a measurable set, we define its \emph{capacity} in $\R^N$ as
\[
\cp(E,\R^N):=\inf\left\{\int_{\R^N}|\nabla u|^2+u^2\,dx:u\in H^1(\R^N),\ u\geq 1\mbox{ in a neighborhood of }E\right\}.
\]
We say that a property holds \emph{quasi-everywhere} (q.e.) if it holds at any point $x$, except at most a set of zero capacity.
\end{definition}
Notice that a set can have positive capacity but zero Lebesgue measure, an easy example being a segment in $\R^2$, thus a property can hold a.e. but \emph{not} q.e.. On the other hand, a set of zero capacity has also zero Lebesgue measure.

\begin{definition}
We say that a set $D\subset\R^N$ is \emph{quasi-open} if for all $\eps>0$ there exists an open set $D_\eps\supset D$ such that $\cp(D_\eps\setminus D)\leq \eps$.
We say that a function $f\colon D\rightarrow \R$ is \emph{quasi-continuous} if for all $\eps>0$ there exists an open set $D_\eps\subset D$ such that $\cp(D\setminus D_\eps)\leq \eps$ and the restriction of $f$ to $D_\eps$ is continuous.
\end{definition}

It is standard to see (a classical reference is~\cite{evansgariepy}) that any function $u\in H^1(\Om)$ admits a quasi-continuous representative $\widehat u$, which is unique up to sets of zero capacity. Moreover it can be pointwise characterized as \[
\widehat u(x)=\lim_{r\rightarrow 0}\frac{1}{|B_r|}\int_{B_r(x)}u(y)\,dy.
\]
From now on, we identify any $H^1$ function with its quasi-continuous representative. Notice that, given $u\in H^1(\R^N)$, then the superlevel set $\{u>0\}$ is quasi-open and vice-versa for any quasi-open set $D$, there is a function $u \in H^1(\R^N)$ such that $D=\{u>0\}$ up to sets of zero capacity (see~\cite[Chapter 3]{hp}).

We are now in position to introduce two Sobolev spaces suitable for dealing with mixed Dirichlet-Neumann eigenvalues, following~\cite{buve}.
\begin{definition}
Let $\Om\subset\R^N$ be a Lipschitz domain and $D\subset\Om$ be a quasi-open set. We define two closed linear subspaces of $H^1(\Om)$ as 
\[
\begin{split}
H^1_0(D,\Om)&:=\left\{u\in H^1(\Om):u=0\;\mbox{q.e. in }\Om\setminus D\right\},\\
\widetilde H^1_0(D,\Om)&:=\left\{u\in H^1(\Om):u=0\;\mbox{a.e. in }\Om\setminus D\right\}
\end{split}
\]
(in particular, the former is closed because, according 
to~\cite[Proposition~3.3.33]{hp}, if $f_n\to f$ in $H^1(\Om)$ then $f_n\rightarrow f$ 
q.e., up to a subsequence). 
\end{definition}
We stress that, if $D\subset \Om$ is open and Lipschitz, then the spaces $H^1_0(D,\Om)$ and $\widetilde H^1_0(D,\Om)$ coincide, otherwise there is only the inclusion $H^1_0(D,\Om)\subset \widetilde H^1_0(D,\Om)$. In order to visualize that the inclusion can actually be strict one can consider $\Om=\R^N$ and $D=B_1(0)\setminus [0,1]\times \{0\} $.
Moreover, if $\Om=\R^N$ and $D$ is open and Lipschitz, then $\widetilde H^1_0(D,\Om)=H^1_0(D,\R^N)=H^1_0(D)$, that is the completion of $C^\infty_c(D)$ with respect to the $\|\cdot\|_{H^1}$ norm, which is the usual definition of Sobolev space.

First of all we need to specify what the meaning of solving a PDE in these spaces actually is.
Given $\Om\subset \R^N$ a Lipschitz domain and $D\subset \Om$ quasi-open, we say that, for any $f\in L^2(\Om)$, $u$ solves the problem
\begin{equation}\label{eq:defpde}
-\Delta u=f,\qquad u=0\;\mbox{on }\partial D\cap \Om,\qquad \partial_\nu u=0\;\mbox{on }\partial D\cap \partial \Om,
\end{equation}
if $u\in H^1_0(D,\Om)$ and \[
\int_{\Om}\nabla u\cdot \nabla v\,dx=\int_{\Om}fv\,dx,\qquad \forall\;v\in H^1_0(D,\Om).
\]
Then, as soon as $\Om\subset\R^N$ is a bounded Lipschitz domain and $D\subset\Om$ is quasi-open with $|D|<|\Om|$, the inclusion $H^1_0(D,\Om)\hookrightarrow L^2(\Om)$ is compact.
Thus, for all $f\in L^2(\Om)$, there is a unique minimizer $w_f\in H^1_0(D,\Om)$ for the functional\[
H^1_0(D,\Om)\ni v\mapsto J_f(v):=\frac{1}{2}\int_{\Om}|\nabla v|^2\,dx-\int_{\Om}vf\,dx,
\]
and the Euler-Lagrange equation for $w_f$ corresponds to~\eqref{eq:defpde}.
The special case $f=1$ is very important. We denote by $w_1=w_D\in H^1_0(D,\Om)$ the minimizer of the functional, \[
H^1_0(D,\Om)\ni v\mapsto \frac12\int_{\Om}|\nabla v|^2\,dx-\int_{\Om}v\,dx,
\]
which is usually called \emph{torsion function} and solves the PDE
\begin{equation}\label{eq:torsione}
-\Delta w_D=1,\qquad w_D=0\;\mbox{on }\partial D\cap \Om,\qquad \partial_\nu w_D=0\;\mbox{on }\partial D\cap \partial \Om.
\end{equation}
It is then possible to prove (\cite[Proposition 2.7, (5)]{buve}) that $H^1_0(D,\Om)=H^1_0(\{w_D>0\},\Om)$.

The torsion function allows us to define a notion of convergence of sets, see~\cite[Section 3]{buve}.
\begin{definition}\label{de:gammaconv}
Let $\Om\subset \R^N$ be a bounded Lipschitz domain and $D_n,D\subset \Om$ be quasi-open sets.
We say that $D_n$ \emph{$\gamma$-converges} to $D$ if $w_{D_n}\rightarrow w_D$ strongly in $L^2(\Om)$ (see \eqref{eq:torsione}).
We say that $D_n$ \emph{weakly $\gamma$-converges} to $D$ if $w_{D_n}\rightarrow w$ strongly in $L^2(\Om)$ for some function $w\in H^1(\Om)$ and $D=\{w>0\}$ q.e..
\end{definition}
As soon as the inclusion $H^1_0(D,\Om)\hookrightarrow L^2(\Om)$ is compact, then also the first eigenvalue of the mixed Dirichlet-Neumann Laplacian is well defined (see~\cite[Remark 2.2]{buve}):\[
\mu(D,\Om)=\min{\left\{\int_{\Om}|\nabla v|^2\,dx:v\in H^1_0(D,\Om),\;\int_{\Om}v^2\,dx=1\right\}},
\]
and it is finite and strictly positive. The main properties of the first eigenvalue for the mixed Dirichlet-Neumann Laplacian are the same as in the case of the Dirichlet-Laplacian:
\begin{itemize}
\item The first eigenfunction (normalized in $L^2$) is denoted by $u$ and is (chosen) non-negative, therefore $\mu(D,\Om)$ is a simple eigenvalue if $D$ is connected.
\item The eigenvalue is monotone with respect to set inclusion: if (q.e.) $D_1\subset D_2\subset\Om$, then $\mu(D_2,\Om)\leq \mu(D_1,\Om)$. This follows from the inclusion of the Sobolev spaces $H^1_0(D_1,\Om)\supset H^1_0(D_2,\Om)$.
\item If a quasi-open set $D$ is the disjoint union of $D_1,D_2$ (that is, $\cp(D_1\cap D_2)=0$ and $D=D_1\cup D_2$), then 
\[
\mu(D,\Om)=\min\Big\{\mu(D_1,\Om),\mu(D_2,\Om)\Big\}.
\]
\end{itemize}
The reason of the importance of the (weak-)$\gamma$-convergence is that eigenvalues of the mixed Dirichlet-Neumann Laplacian and the measure functional are lower-semicontinuous with respect to it, see~\cite[Proposition 3.12]{buve}.

In the case of $\Om=\R^N$ and $D$ quasi-open, the (less common) Sobolev-like spaces $\widetilde H^1_0$ have been treated in~\cite[Section 2]{bmpv}, but we recall here the main features since we are working in a slightly different setting.
We need first to give meaning to some of the quantities above also in the Sobolev-like space $\widetilde H^1_0(D,\Om)$. We say (following~\cite[Section 2]{bmpv}) that, for any $f\in L^2(\Om)$, $\widetilde u$ solves in $\widetilde H^1_0(D,\Om)$ the problem 
\[
-\Delta \widetilde u=f,\qquad \widetilde u=0\;\mbox{on }\partial D\cap \Om,\qquad \partial_\nu \widetilde u=0\;\mbox{on }\partial D\cap \partial \Om,
\]
if $\widetilde u\in \widetilde H^1_0(D,\Om)$ and \[
\int_{\Om}\nabla \widetilde u\cdot \nabla v\,dx=\int_{\Om}fv\,dx,\qquad \forall\;v\in \widetilde H^1_0(D,\Om).
\]
As before, if $\Om\subset \R^N$ is a bounded Lipschitz domain and $D\subset\Om$ is 
quasi-open, with $|D|<|\Om|$, the inclusion $\widetilde H^1_0(D,\Om)\hookrightarrow L^2(\Om)$ is compact. Thus we can define the (Lebesgue) torsion function $\widetilde w_D\in \widetilde H^1_0(D,\Om)$, which is the unique minimizer of the functional
\[
\widetilde H^1_0(D,\Om)\ni v\mapsto \frac12\int_{\Om}|\nabla v|^2\,dx-\int_{\Om}v\,dx.
\]

We note that in the framework of Sobolev-like spaces $\widetilde H^1_0$ a weak maximum principle still holds, and $D_1\subset D_2$ a.e. implies that $\widetilde w_{D_1}\leq \widetilde w_{D_2}$ a.e. (see also the proof of Lemma \ref{le:spaceHtilde}).

In the following we need a slight generalization of~\cite[Lemme~3.3.30]{hp}.
\begin{lemma}\label{le:omegaDchiuso}
Let $A\subset \R^N$ be an open set and $w$ be a quasi-continuous function such that $w\geq 0$ a.e. in $A$, then $w\geq 0$ q.e. in $A$.
Moreover, if $w=0$ a.e. in $A$ then $w=0$ q.e. in $A$.
\end{lemma}
\begin{proof}
For any open set $E\subset \overline E\subset A$, we can find a smooth cutoff function $\varphi\in C^\infty_c(\R^N,[0,1])$ such that $\varphi=1$ in $E$ and $\varphi=0$ in $\R^N\setminus A$. Then $w\varphi\colon \R^N\rightarrow \R$ is quasi-continuous, $w\varphi\geq 0$ a.e. in $\R^N$ and by~\cite[Lemme~3.3.30]{hp}, we infer that $w\varphi\geq 0$ q.e. in $\R^N$, thus $w\geq 0$ q.e. in $E$. Moreover, since $A$ is an open set, we have that $w\geq 0$ q.e. in the whole $A$ by an exhaustion procedure: since for every $x_{0}\in A$ there exists $B_{r(x_0)}(x_{0})$ such that 
$\overline{B_{r(x_0)}(x_{0})}\subset A$ and $w\geq 0$ quasi-everywhere in 
$B_{r(x_0)}(x_{0})$, then $w\geq 0$ quasi-everywhere in $\bigcup_{x_0\in A} B_{r(x_0)}(x_0)=A$.
The second part of the statement follows simply by applying the first part to $\pm w$.
\end{proof}
The next lemma provides more insight in the relation between the spaces $H^1_0$ and $\widetilde H^1_0$. This is a well-known property, and the proof can be done as in~\cite[Proposition 4.7]{deve}.
\begin{lemma}\label{le:spaceHtilde}
Let $\Om\subset\R^N$ be a Lipschitz domain, $D\subset \Om$ be quasi-open. The set $\omega_D:=\{\widetilde w_D>0\}$ is quasi-open, contained a.e. in $D$, and such that $\widetilde H^1_0(D,\Om)=\widetilde H^1_0(\omega_D,\Om)=H^1_0(\omega_D,\Om)$. Moreover $\omega_D$ is contained q.e. in $\overline D$.
\end{lemma}
We summarize in the next theorem some results for the spectral drop problem, obtained  in~\cite[Theorem 4.1, Remark 4.4]{buve}. 
\begin{theorem}[Buttazzo-Velichkov]\label{thm:buve}
Let $\Om\subset\R^N$ be a bounded Lipschitz domain, and let $0<\delta<|\Om|$.
Then $\sd(\delta)$ is achieved (recall~\eqref{eq:def_sd}), and any optimal set
$D^*$ is connected, and thus $\mu(D^*,\Om)$ is simple. Furthermore
the boundary of $D^*$ intersects $\partial \Om$ orthogonally, if the intersection lies on a regular point of $\partial \Om$.
\end{theorem}
\begin{remark}\label{rem:minoreouguale}
Thanks to the monotonicity of the eigenvalues, it is possible to see, as in Remark~\ref{rem:monot}, that it is equivalent to consider problem~\eqref{eq:def_sd} with the constraint on the measure $|D|\leq \delta$ instead of the equality constraint, that is,\[
\sd(\delta)=\min{\left\{\mu(D,\Om):D\subset \Om,\;\mbox{quasi-open, }|D|\leq\delta\right\}}
\]
\end{remark}

Thanks also to Lemma~\ref{le:spaceHtilde}, we have this crucial corollary.
\begin{corollary}\label{cor:eqmeascap}
Let $\Om\subset\R^N$ be a Lipschitz domain, $\delta\in (0,|\Om|)$, then
\begin{equation}\label{eq:equiv}
\sd(\delta)=\min{\left\{\mu(\omega_D,\Om):D\subset \Om,\;\mbox{quasi-open, }|D|\leq\delta\right\}}.
\end{equation}
\end{corollary}
\begin{proof}
First of all we note that, having in mind also Lemma~\ref{le:spaceHtilde}, for any quasi-open set $D\subset \Om$, $H^1_0(\omega_D,\Om)=\widetilde H^1_0(D,\Om)\supset H^1_0(D,\Om)$, thus $\mu(D,\Om)\geq \mu(\omega_D,\Om)$, by definition of the first Dirichlet-Neumann eigenvalue. Taking into account also Remark~\ref{rem:minoreouguale}, this implies that, in~\eqref{eq:equiv}, the left hand side is greater than or equal to the right hand side.
On the other hand, this inequality can not be strict since, by Lemma~\ref{le:spaceHtilde}, $\widetilde H^1_0(\omega_D,\Om)=H^1_0(\omega_D,\Om)$ and $|\omega_D|\leq |D|\leq \delta$, thus $\omega_D$ is admissible in the minimization on the left hand side.
\end{proof}
\begin{remark}
A natural question concerns the regularity properties of the free boundary 
$\partial D^*\cap \Om$ of an optimal set. Many results are available in the case of 
full Dirichlet boundary conditions, see for instance~\cite{brla,bhp,bmpv,mtv}:
in this case, any optimal set $D^*$ is open, and $\partial D^*\cap \Om$ 
is locally smooth and analytic up to a singular part, whose Hausdorff dimension is less than or 
equal to $N-N^*$, for a critical dimension $N^*\in\{5,6,7\}$ (see~\cite{jesa,mtv} and the 
references therein for more details). In particular, in dimension $N\le4$ the singular part 
is empty and the free boundary is analytic. 
Since the approach of all the aforementioned papers is local, we expect that
the same regularity properties should hold also in the case of mixed Neumann-Dirichlet conditions. Anyway, this falls outside the scope of this paper, therefore we will not pursue this line here.
\end{remark}
\begin{remark}\label{rem:tocca il bordo}
Another remarkable property for optimal sets $D^*$ associated to $\sd(\delta)$ is that they must touch the boundary of $\Om$, more precisely $\mathcal H^{N-1}(\partial D^*\cap \partial \Om)>0$. This is treated in~\cite[Remark 4.3]{buve} and holds at least if $\Om$ is smooth and if $N<N^*$ (and hence for $N\leq 4$), since one needs to know that the one phase free boundary for the Bernoulli problem is smooth and analytic (see the above discussion and~\cite{jesa}) in order to apply the argument by Buttazzo and Velichkov.
\end{remark}
The last piece of information we need is the optimality condition for the free 
boundary $\partial D^*\cap \Om$ of an optimal set. Again, when 
only Dirichlet boundary conditions are involved this condition follows by a standard shape 
derivative argument (see for example~\cite[Lemma~2.8]{dl} or~\cite[Chapter~5]{hp}). Actually, 
nothing changes in our situation, since the free boundary has still Dirichlet conditions and the argument is local. 
\begin{lemma}\label{le:optimalitycond}
Let $\Om\subset\R^N$, $\delta\in(0,|\Om|)$, and $D^*$ be an optimal set achieving
$\sd(\delta)$.
Assume moreover that $\Gamma$, a non-empty relatively open subset of $\partial D^*\cap \Om$, is smooth. Then, for any vector field $V\in C^{\infty}_c(\Om,\Om)$ which preserves the measure of $D^*$ and with $\partial D^*\cap\supp V\subset \Gamma$, we have
\begin{equation}\label{eq:optcond}
\int_{\partial D^*\cap\supp V}|\nabla u|^2\,V\cdot\nu\,d\mathcal H^{N-1}=0,
\end{equation}
where $u$ denotes as usual the normalized first eigenfunction associated to $\mu(D^*,\Om)$ and $\nu$ is the outer unit normal.
Moreover, it follows from~\eqref{eq:optcond} that there is a constant $c>0$ such that $|\nabla u|^2=c$ on $\Gamma$.
\end{lemma}
\begin{remark}
The hypothesis of regularity of at least a relatively open part of the free boundary in the statement of Lemma~\ref{le:optimalitycond} is necessary for proving an optimality condition in a classic sense. In the last years, in the study of free boundary problems, it has been shown how to prove the same condition, in a weaker sense, without regularity assumptions. Two possible ways are to consider $\Delta u$ as a measure concentrated on the boundary (see~\cite{alca}), or to use the viscosity solutions approach (see~\cite{desilva,mtv}).
\end{remark}
Lemma \ref{le:optimalitycond} allows to complete the proof of Proposition
\ref{prop:henrotoudet}.
\begin{proof}[Proof of Proposition \ref{prop:henrotoudet}]
Let $\Gamma\subset\partial D^*\cap\Om$ be a non-empty, relatively open portion of sphere, centered at the origin.
Then, by Lemma \ref{le:optimalitycond}, the first eigenfunction $u\in H^1_0(D^*,\Om)$,
associated to $\mu(D^*,\Om)$, satisfies
\[
\left\{
\begin{aligned}
-\Delta u&=\mu(D^*,\Om) u,\qquad &&\mbox{in }\Om,\\
u&=0,\qquad &&\mbox{in }\Om\setminus D^*,\\
|\partial_\nu u |^2&=c,\qquad &&\mbox{on }\Gamma.
\end{aligned}
\right.
\]
To start with, we prove that $u$ is radially symmetric.
Following an idea from~\cite{ho,llnp}, we consider the function $v_{ij}=x_i\partial_j u -x_j\partial_i u $ for $i\not=j$, which solves the problem
\[
\left\{
\begin{aligned}
-\Delta v_{ij}&=\mu(D^*,\Om) v_{ij},\qquad &&\mbox{in }\Om,\\
v_{ij}&=0,\qquad &&\mbox{in }\Om\setminus D^*,\\
\partial_\nu v_{ij}&=0,\qquad &&\mbox{on }\Gamma.
\end{aligned}
\right.
\]
The first two conditions are immediate to verify, for the third one it is possible to 
check, since $ u \in C^\infty(D^*)$ and the normal to $\Gamma$ is $x/|x|$, that
\[
|x|\partial_\nu v_{ij}=|x|\sum_{k}x_k\partial_kv_{ij}=\sum_k(x_i\partial_j-x_j\partial_i)x_k\partial_k u +v_{ij}=|x|(x_i\partial_j-x_j\partial_i)\partial_{\nu} u =0,
\]
because $v_{ij}=0$ and $\partial_\nu u $ is constant on $\Gamma$. Now, we can use the Cauchy-Kovaleskaya Theorem to deduce that $v_{ij}=0$ in a neighborhood of $\Gamma$ and
thus in the whole $D^*$ by unique continuation.
We have proved that $x_i\partial_j  u =x_j\partial_i u $ for $i\not=j$, thus $ u (x) = w(|x|)$ is radially symmetric in $D^*$. Moreover it is regular up to any regular part of
$\partial D^*$, and $\nabla u (x) = w'(|x|) x/|x|$. (Since $u\equiv0$ in $\Omega\setminus D^*$, then $u$ is actually radial in the whole $\Om$, although it is not $C^1$ across
$\partial D^*\cap\Omega$.)

Now, take $\Gamma'\subset\partial D^*\cap \Omega$ a connected, regular surface;
then $u|_{\Gamma'}$ is constant, thus $|\partial_\nu u| = |\nabla u|$, i.e. $\nu(x) = \pm x/|x|$ on $\Gamma'$. Elementary arguments show that $\Gamma'$
is a portion of sphere centered at $0$.

Similarly, let $\Gamma'\subset \partial D^*\cap \partial\Omega$ be connected and regular. Then the Neumann condition is satisfied pointwise on $\Gamma'$, and $u|_{\Gamma'}>0$
by Hopf lemma. We obtain that $w'(|x|) x\cdot \nu(x) \equiv 0 $ on $\Gamma'$. On the relatively open $\gamma_1\subset\Gamma'$ where  $x\cdot \nu(x)\neq0$, we have that
$w'(|x|)\equiv0$ on $\gamma_1$, so that $w(|x|)$ is constant on (each connected component of) $\gamma_1$. Indeed, using the equation and the regularity up to the boundary, we have that
zeroes of $w'$ corresponding to positive values of $w$ are isolated. Finally, if $w'(|x|)\neq 0$ on $\gamma_2\subset\Gamma'$, then $ x\cdot\nu(x) \equiv 0$ and, again by elementary arguments, we conclude that $\gamma_2$ is a disjoint union of portions of cones with vertex at $0$. Since no such $\gamma_1$ and $\gamma_2$ can be joined in a regular way, we deduce that one of them is empty, concluding the proof.
\end{proof}

\section{Asymptotic analysis as \texorpdfstring{$\beta\rightarrow \infty$}{b->infinity}}\label{sec:beta2infty}

In this section we will perform our asymptotic analysis of Problem \eqref{eq:def_od}, providing the proof of Theorem \ref{thm:equality}, as
well as an improved version of Lemma \ref{lem:intro_beta2infty}, in the more general setting of quasi-open sets. In order to do this, let us first fix some notations.

Throughout this section, $\Om\subset \R^N$ is a bounded Lipschitz domain,
$\delta\in(0,|\Om|)$ is fixed,  while $\beta > \dfrac{\delta}{|\Om|-\delta}$.
Let  $D\subset \Om$ be any fixed quasi-open set with $|D|=\delta$, so that $m_\beta := 
(1+\beta)\ind{D} - \beta$ is admissible for the minimization of $\od(\beta,\lambda)$, 
and $\lambda(\beta,D) :=\lambda(m_\beta)$ is achieved by $u_\beta$, which solves
\begin{equation}\label{eq:ubeta}
\begin{cases}
-\Delta u_\beta = \lambda(\beta,D) m_\beta u_\beta &\text{in }
 \Omega,
  \\
\|u_{\beta}\|_{L^{2}(D)}=1,\,\partial_\nu u_\beta=0 &\text{on } \partial\Omega.
\end{cases}
\end{equation}
Turning to
\begin{equation}\label{eq:pmin}
\od(\beta,\delta)=\min\{\lambda(\beta,D):D\subset \Om,\ |D|=\delta\},
\end{equation}
it is important to note that the above minimization can be equivalently performed among open or among quasi-open or even among measurable sets, since optimal sets can be chosen
to be open (see Theorem~\ref{thm:base}).
In the following, we perform the minimization in the class of quasi-open sets because this is the suitable class of sets in which we can work when dealing with the spectral drop problem.

We will prove Theorem \ref{thm:equality}  through a sequence of lemmas. 
We recall that, according to Lemma~\ref{le:spaceHtilde}, for any quasi-open set $D\subset \Om$, there exists another quasi-open set $\omega_D\subset D$ a.e. such that $\widetilde H^1_0(D,\Om)=H^1_0(\omega_D,\Om)$.
\begin{lemma}\label{le:stime}
Let $\beta>1$, $D\subset\Om$ be a quasi-open set with $|D|=\delta\in(0,|\Om|)$. The following conclusions hold.
\begin{enumerate}
\item
\(\dys
0< \beta\int_{\Omega\setminus D} u_\beta^2\,dx
\leq \int_{D} u_\beta^2\,dx = 1.
\)
\item
\(0< \lambda(\beta,D)\leq\mu(\omega_D,\Om)\leq \mu(D,\Om)\).
\item
\(\|u_{\beta}\|_{H^{1}(\Omega)}\leq (2+\mu(\omega_D,\Om))^{1/2}\)
\end{enumerate}
\end{lemma}
\begin{proof}
The first point is a direct consequence of the normalization we chose for $u_\beta$, together
with the fact that
\[
\into m_{\beta }u_{\beta}^{2}\,dx>0.
\]
In order to show the second part of the statement,
let $u\in H^1_0(D,\Om)$ denote the  eigenfunction associated to
$\mu(D,\Omega)$. Then
\[
\into m_{\beta}u^{2}\,dx=\int_{D} m_{\beta}u^{2}\,dx=
\int_{D} u^{2}\,dx=1>0\,,
\]
as $u$ is also normalized so that it has unit $L^{2}(D)$ norm.
As a consequence, \(u\) is an admissible competitor in the minimization problem defining $\lambda(\beta,D)$, thus
\[
\lambda(\beta,D)\leq \dfrac{\|\nabla u\|_{L^{2}(\Omega)}}{\int_{\Omega}m_{\beta}u^{2}}
= \dfrac{\|\nabla u\|_{L^{2}(\Om)}^{2}}{\|u\|_{L^{2}(\Om)}^{2}}=\mu(D,\Om).
\]
We can actually say something more: exploiting Lemma~\ref{le:spaceHtilde}, we call $\widetilde u\in \widetilde H^1_0(D,\Om)=H^1_0(\omega_D,\Om)\subset H^1(\Om)$ the first eigenfunction with unit $L^2$ norm associated to $\mu(\omega_D,\Om)$. Since $\widetilde u=0$ a.e. in $\Om\setminus D$, we can repeat the above argument and obtain
\[
\lambda(\beta,D)\leq \dfrac{\|\nabla \widetilde u\|_{L^{2}(\Omega)}}{\int_{\Omega}m_{\beta}\widetilde u^{2}}
= \dfrac{\|\nabla \widetilde u\|_{L^{2}(\Om)}^{2}}{\|\widetilde u\|_{L^{2}(\Om)}^{2}}=\mu(\omega_D,\Om)\leq \mu(D,\Om),
\]
where the last inequality follows since $H^1_0(D,\Om)\subset H^1_0(\omega_D,\Om)$.
Finally, part $3$ follows using part $2$ and the normalization of $u_{\beta}$, as it results
\[
\int_{\Om}|\nabla u_\beta|^2\,dx=\lambda(\beta,D) \int_{D_\beta}(u_\beta)^2\,dx-\lambda(\beta,D)\beta\int_{\Om\setminus D_\beta}(u_\beta)^2\,dx
\leq \lambda(\beta,D)
\leq \mu(\omega_D,\Om),
\]
and, as $\beta>1$,
\[
\int_{\Om}u_\beta^2\,dx=\int_{D}u_\beta^2\,dx+\int_{\Om\setminus D}u_\beta^2\,dx \le 1+ \frac{1}{\beta}<2. \qedhere
\]
\end{proof}
\begin{lemma}\label{le:lim}
Let $D\subset\Om$ be a quasi-open set with $|D|=\delta\in(0,|\Om|)$. Then,
the  sequence $u_{\beta}$ strongly converges as $\beta\rightarrow \infty$ in $H^{1}(\Omega)$ to $\widetilde u\in \widetilde H^1_0(D,\Om)=H^1_0(\omega_D,\Om)$, which achieves $\mu(\omega_D,\Om)$. In particular
\begin{equation}\label{eq:dislabmu1}
\mu(\omega_D,\Om)= \lim_{\beta\rightarrow \infty}\lambda(\beta,D).
\end{equation}
\end{lemma}
\begin{proof}
Lemma~\ref{le:stime}  implies that there exists $\widetilde u\in H^{1}(\Omega)$ such that
\(u_{\beta}\)  converges to \(\widetilde u\) weakly in \( H^{1}(\Omega)\) and, up to a subsequence,  strongly in \( L^{2}(\Omega)\) and  almost everywhere.
As a consequence, $\widetilde u\geq 0$ in $\Omega$ a.e. (and also q.e. by Lemma~\ref{le:omegaDchiuso}), and
\begin{equation}\label{eq:unorm}
\int_{D}\widetilde u^{2}\,dx=1,
\end{equation}
so that $\widetilde u\not\equiv 0$ a.e.. On the other hand, from conclusion $(1)$ of Lemma~\ref{le:stime} it follows that
\[
\int_{\Omega\setminus D}\widetilde u^{2}\,dx=\lim_{\beta\to+\infty}\int_{\Omega\setminus D}u_{\beta}^{2}\,dx\leq \dfrac1{\beta}\to 0\,,
\]
so that \(\widetilde u\equiv 0\) a.e. in \(\Omega\setminus D\) and thus $\widetilde u \in \widetilde H^1_0(D,\Om)=H^1_0(\omega_D,\Om)$. Moreover, as
$u_{\beta}$ solves \eqref{eq:ubeta}, it results
\begin{equation}\label{eq:eqtest}
\int_{\Om} \nabla u_{\beta}\cdot\nabla \eta\,dx =\lambda(\beta,D)\int_{\Om}u_{\beta}\eta\,dx\,,\qquad \forall\,\eta\in H^{1}(\Omega),
\end{equation}
and, observing that  $\lambda(\beta,D)\to \lambda$ for some
\(\lambda\in [0,\mu(\omega_D,\Om)]\) we can pass to the limit, so that
\[
\int_{\Om} \nabla\widetilde u \cdot\nabla \eta\,dx =\lambda \int_{\Om}\widetilde u \eta\,dx,\qquad \forall\,\eta\in  H^{1}_0(\omega_D,\Omega)\subset H^1(\Om).
\]
So we have that $\widetilde u$ solves the problem \[
-\Delta \widetilde u=\lambda \widetilde u,\qquad \widetilde u\in  H^1_0(\omega_D,\Om),
\]
and moreover it is a competitor in the minimization defining $\mu(\omega_D,\Om)$: \[
\mu(\omega_D,\Om)\leq \frac{\|\nabla\widetilde u\|^2_{L^2(\Om)}}{\|\widetilde u\|^2_{L^2(\Om)}}=\lambda \leq \mu(\omega_D,\Om),
\]
thus it is an eigenfunction with eigenvalue $\lambda= \mu(\omega_D,\Om)$, that is, the first eigenfunction.

In order to prove that the convergence $u_\beta\rightarrow \widetilde u$ in $H^1(\Om)$ is actually strong it is enough to demonstrate the convergence of the $L^2$ norm of the gradients. For showing this, we choose $v=u_\beta\in H^1(\Om)$ in~\eqref{eq:eqtest}, and obtain \[
\int_{\Om}|\nabla u_\beta|^2\,dx=\lambda(\beta,D)\int_{\Om}|u_\beta|^2\,dx\rightarrow \mu(\omega_D,\Om)\int_{\Om}|\widetilde u|^2\,dx=\int_{\Om}|\nabla \widetilde u|^2\,dx.
\]

Concerning the last part of the statement, it is enough to use the strong $H^1$ convergence $u_\beta\rightarrow \widetilde u$ as $\beta\rightarrow \infty$ and the fact that $\int_\Om u_\beta^2\,dx\geq \int_D u_\beta^2\,dx-\beta\int_{\Om\setminus D}u_\beta^2\,dx$, to show:
\[
\mu(\omega_D,\Om)=
\dfrac{\int_{\Om}|\nabla \widetilde u|^2\,dx}{\int_\Om \widetilde u^2\,dx}=
\lim_{\beta\rightarrow \infty} \dfrac{\int_{\Om}|\nabla u_\beta|^2\,dx}{\int_\Om u_\beta^2\,dx}
\leq
\lim_{\beta\rightarrow \infty} \dfrac{\int_{\Om}|\nabla u_\beta|^2\,dx}{\int_\Om m_{\beta} u_\beta^2\,dx}
=\lim_{\beta\rightarrow \infty}\lambda(\beta,D),
\]
while the other inequality is immediate from part 2 of Lemma~\ref{le:stime}.
\end{proof}
\begin{proof}[Proof of Lemma \ref{lem:intro_beta2infty}]
It follows from Lemma \ref{le:lim} taking into account that, in case $D$ is an open, Lipschitz set, then  $\omega_D=D$ and $H^1_0(D,\Om)$ and $\widetilde H^1_0(D,\Om)$ coincide.
\end{proof}
The above lemmas allow to control $\od(+\infty,\delta)$ from above, in terms of 
$\sd(\delta)$. The opposite inequality is a bit less straightforward, and to obtain it 
we need ``an $\eps$ of room'' more.
\begin{lemma}\label{le:lim*}
For every $\beta>\frac{\delta}{|\Om|-\delta}$ and $\eps\in \left(\frac{\delta}{\beta}, |\Omega|-\delta\right) $.
\[
\sd(\delta+\eps)\left(1-\sqrt{\frac{\delta}{\eps\beta}}\right)^{2}\leq \od(\beta,\delta).
\]
In particular, if $\eps\in \left(0, |\Omega|-\delta\right) $, $\sd(\delta+\eps) \le \liminf_{\beta\to\infty} \od(\beta,\delta).$
\end{lemma}
\begin{proof}
Let $\eps\in(0, |\Om|-\delta)$ be fixed. For $\beta>1$, let $u^*_\beta$ be an eigenfunction associated to $\lambda^*_\beta =  \od(\beta,\delta)$. By Theorem
\ref{thm:base} we know that $D^*_\beta = \{x:u^*_\beta(x) \ge \ell_\beta\}$, for some $\ell_\beta$, with $|D^*_\beta|=\delta$ and that  $|\{x:u^*_\beta(x) = t\}| = 0$
for every $t$. We deduce the existence of a unique $t_\beta\in(0,\ell_\beta)$ for which
\[
E_\beta := \{x:u^*_\beta(x) > t_\beta\} \qquad\text{satisfies}\qquad |E_\beta|=\delta + \eps.
\]
In particular, $D^*_\beta \subset E_\beta$ and
\[
\eps t_\beta^2 \le \int_{E_\beta\setminus D^*_\beta} (u^*_\beta)^2 \,dx\le \int_{\Omega\setminus D^*_\beta} (u^*_\beta)^2\,dx \le \frac{1}{\beta},
\]
which forces $t_\beta \to 0 $ as $\beta\to\infty$. Let $v=(u^*_\beta - t_\beta)^+$. Then $v \in H^1_0(E_\beta,\Omega)$, and
\[
\mu(E_\beta,\Om) \le \frac{\int_{E_\beta}|\nabla v|^2\,dx}{\int_{E_\beta}v^2\,dx} \le \frac{\int_{\Omega}|\nabla u^*_\beta|^2\,dx}{\int_{D^*_\beta}(u^*_\beta - t_\beta)^2\,dx}
= \lambda^*_\beta \frac{\int_{D^*_\beta}(u^*_\beta)^2\,dx - \beta\int_{\Omega\setminus D^*_\beta}(u^*_\beta)^2\,dx}{\int_{D^*_\beta}(u^*_\beta)^2\,dx-2t_\beta\int_{D^*_\beta}u^*_\beta\,dx +  \delta t_\beta^2}.
\]
Then we note, using also H\"older inequality on $\int_{D^*_\beta}u^*_\beta\,dx$, that
\[
\sd(\delta+\eps)\leq \mu(E_{\beta},\Omega)\leq \frac{\od(\beta,\delta)}{1-2t_{\beta}\delta^{1/2}+\delta{t_{\beta}}^{2}}=\frac{\od(\beta,\delta)}{(1-\delta^{1/2}t_\beta)^2},
\]
which immediately yields the conclusion as soon as $\eps>\delta/\beta$ so that $t_{\beta}\leq ( \eps\beta)^{-1/2}$.
\end{proof}
%
%
%
We are now in position to prove Theorem~\ref{thm:equality}.
\begin{proof}[Proof of Theorem \ref{thm:equality}]
The first part of the theorem follows from Lemma \ref{le:lim*}. To prove the second part, we observe that, from  the definition of $\lambda(\beta,D)$, part 2 of Lemma~\ref{le:stime} and Lemma~\ref{le:lim}, keeping in mind also Corollary~\ref{cor:eqmeascap}, we have
\begin{equation}\label{eq:disla}
\begin{split}
\lim_{\beta\to+\infty}\od(\beta,\delta)=&\lim_{\beta\to+\infty}\min\Big\{\lambda_{1}((1+\beta)\ind{D} - \beta):D\subset \Om,\;\mbox{quasi-open, }|D|=\delta\Big\}\\
&\leq \lim_{\beta\to+\infty}\min\Big\{\mu(\omega_D,\Om):D\subset \Om,\;\mbox{quasi-open, }|D|=\delta\Big\}\\
&=\min\Big\{\mu(\omega_D,\Om):D\subset \Om,\;\mbox{quasi-open, }|D|=\delta\Big\}=\sd(\delta).
\end{split}
\end{equation}
On the other hand, let $E^*_n$ be a minimizer associated to the problem $\sd(\delta + \eps_n)$,
with $\eps_n\to 0^+$ as $n\to\infty$. Then, by Lemma \ref{le:lim*},
\[
|E^*_n| = \delta + \eps_n
\qquad\text{and}\qquad
\mu(E^*_n,\Om) \le \liminf_{\beta\to\infty} \od(\beta,\delta).
\]
At this point, having in mind Definition~\ref{de:gammaconv} of weak $\gamma$-convergence, we can use \cite[Proposition 2.3,(3) and Proposition 2.7,(1)]{buve} to infer that $E^*_n$ weakly $\gamma$-converges to some quasi-open set $F$. In turn, \cite[Prop. 3.12]{buve} implies
\[
|F| \le \liminf_n |E^*_n| = \delta,
\qquad
\sd(\delta)\le\mu(F,\Om) \le \liminf_n \mu(E^*_n,\Om) \le \liminf_{\beta\to\infty} \od(\beta,\delta)
\]
(recall Remark \ref{rem:minoreouguale}). This shows the reverse inequality of
\eqref{eq:disla}, concluding  the proof of Theorem~\ref{thm:equality}.
\end{proof}
\begin{proof}[Proof of Corollary \ref{coro:intersectsboundary}]
Let $D^*$ be an optimal set for $\sd(\delta)$, such that $\Hcal^{N-1}(\partial D^*\cap\partial\Omega) > 0$. 
Since $D$ is quasi-open, 
by Lemma~\ref{le:lim} we have that \[
\lim_{\beta\rightarrow \infty}\la(\beta,D)=\mu(\omega_D,\Om).
\]
Moreover, noting that 
$\mathcal H^{N-1}(  
\overline{D}\cap \partial \Om)=\mathcal H^{N-1}(\partial D\cap \partial \Om)=0$ by hypothesis, Lemma~\ref{le:spaceHtilde} forces $\mathcal H^{N-1}(\partial \omega_D\cap \partial \Om)=0$. Thus, our assumptions entail that 
$\mu(\omega_D,\Omega)>\sd(\delta)$. Finally, applying conclusion 2 of Lemma~\ref{le:stime}, we infer that \[
\la(\beta,D^*)\leq\mu(D^*,\Om)=\sd(\delta)<\mu(\omega_D,\Om)=\lim_{\beta\rightarrow \infty} \la(\beta,D),
\] 
which proves the claim.
\end{proof}

\section{Spherical shapes in the spectral drop problem}\label{sec:asymptspecrtaldrop}

\subsection{Relative isoperimetric inequalities and \texorpdfstring{$\alpha$}{alpha}-symmetrization}

In order to provide an estimate from below of $\mu(D,\Omega)$ we will exploit the
$\alpha$-symmetrization on cones, which was introduced in \cite{MR572958} for planar
domains and then extended in \cite{MR876139} to general dimension.

For any $0<\alpha<\omega_N=|B_1|$, let $\Sigma_\alpha$ denote any open cone, with vertex at
the origin, having the property that
\[
|\Sigma_\alpha\cap B_1| = \frac1N \Hcal^{N-1}(\Sigma_\alpha\cap \partial B_1) = \alpha
\]
(while in \cite{MR876139,MR963504} cones having specific shape are chosen, for our
purposes we need no further property).
Then the $\alpha$-symmetrization of a measurable
set $D\subset\R^N$ is defined as
\[
\Ccal_\alpha(D) := \Sigma_\alpha \cap B_{r(\alpha,|D|)},
\]
where $r(\alpha,|D|)$ is such that $|\Ccal_\alpha(D)|=|D|$ (i.e. $r(\alpha,|D|)=(|D|/\alpha)^{1/N}$).
Consequently, for a measurable, non-negative $u:D\to\R$, we define its
$\alpha$-symmetrization $\Ccal_\alpha u : \Ccal_\alpha(D) \to \R$ as
\[
\Ccal_\alpha u(x) := \sup\left\{t:|\{y:u(y)>t\}|>\alpha|x|^N\right\}.
\]
Then $\Ccal_\alpha u$ is radially decreasing in $0<|x|< r(\alpha,|D|)$, and $|\{u>t\}| =
|\{\Ccal_\alpha u>t\}|$ (actually, defining $\Sigma_{\omega_N} = \R^N$, the above
procedure leads to the usual Schwarz symmetrization).

From now on we restrict our attention on quasi-open sets $D$. Our aim is to show that, for a suitable choice of $\alpha$,
\begin{equation}\label{eq:firstalpha}
\mu(D,\Omega) \ge \mu(\Ccal_\alpha (D),\Sigma_\alpha) = \lambda_1^\text{Dir} \cdot
r^{-2}(\alpha,|D|),
\end{equation}
where $\lambda_1^\text{Dir}$ denotes the first eigenvalue
of the Dirichlet Laplacian in $B_1$:
\[
\begin{cases}
-\Delta \varphi = \lambda_1^\text{Dir}\varphi, &\text{in }B_1,\\
\varphi=0, &\text{on }\partial B_1,
\end{cases}
\]
and $\varphi\in H^1_0(B_1)$ is the first Dirichlet eigenfunction. A useful observation for the sequel is that
\begin{equation}\label{eq:half_Rayleigh}
\lambda_1^\text{Dir} = \frac{\int_{B_1^+}
|\nabla\varphi(x)|^2\,dx}{\int_{B_1^+}\varphi^2(x)
\,dx} = \mu(B^+_1,\R^N_+) = \mu(B_1\cap\Sigma,\Sigma),
\end{equation}
for any cone $\Sigma$ having vertex at the center of $B_1$.

The right choice of $\alpha$ in \eqref{eq:firstalpha} will depend on a suitable
isoperimetric constant. This follows closely some ideas in \cite{MR963504}, even though
our situation is
slightly different: while the domains considered in \cite{MR963504} have the boundary
divided in fixed Neumann and Dirichlet parts, here we need to deal with arbitrary
subsets of $\Omega$ of fixed measure.

More precisely, for $0<\delta<|\Omega|$ we define the relative isoperimetric constant inside $\Om\subset\R^N$, with measure constraint $\delta$, as
\begin{equation}\label{eq:isop_const}
K(\Om,\delta):=\inf{\left\{\frac1N \frac{P(D,\Om)}{|D|^{(N-1)/N}} : D\subset\Om,\ |D|\leq \delta\right\}},
\end{equation}
where $P(D,\Om)$ is the De Giorgi perimeter of $D$ relative to $\Omega$:
\[
P(D,\Om):= \sup\left\{\int_D\diverg F:F\in C^\infty_0(\Omega,\R^N),\ |F|
\le1\right\}
\]
(in particular, in the regular case, $P(D,\Om) = \Hcal^{N-1}(\partial D \cap \Omega)$).
Notice that, taking $D=B_\eps(x_0)\subset\Omega$, with $\eps$ small, one easily obtains
\[
0\le K(\Om,\delta) \le \omega_N^{1/N}.
\]
Moreover $K$ is non-increasing with respect to $\delta$.
\begin{lemma}\label{lem:Kincone}
For any cone $\Sigma_\alpha$ and $r>0$,
\[
\frac1N \frac{P(B_r,\Sigma_\alpha)}{|B_r\cap\Sigma_\alpha|^{(N-1)/N}} = \alpha^{1/N}.
\]
Furthermore, if $\Sigma_\alpha$ is convex,
\[
K(\Sigma_\alpha,\delta) = \alpha^{1/N}.
\]
\end{lemma}
\begin{proof}
The first part follows by direct computations. The second one ---which we state
just for the sake of completeness--- is \cite[Theorem 1.1]{MR1000160}.
\end{proof}
Our key result in this setting is the following.
\begin{proposition}\label{prop:isoper_generale}
If $K(\Omega,\delta)$ is defined as in \eqref{eq:isop_const} and $D\subset \Omega$
quasi-open, $|D|\le\delta$, then
\[
\mu(D,\Om)\ge K^2(\Omega,\delta)\lambda_1^\text{Dir} \cdot |D|^{-2/N},
\]
or, equivalently, setting $\alpha = K^N(\Om,\delta)$,
\[
\mu(D,\Om) \ge \mu(\Ccal_\alpha (D),\Sigma_\alpha) = \lambda_1^\text{Dir} \cdot
r^{-2}(\alpha,|D|).
\]
\end{proposition}
Notice that, while $\Ccal_\alpha (D)$ depends on the choice of the cone $\Sigma_\alpha$,  
$\mu(\Ccal_\alpha (D),\Sigma_\alpha)$ only depends on $\alpha$.
\begin{proof}
The proposition is essentially \cite[Proposition 1.2]{MR963504}, see also
\cite[Example~5.3]{buve}. As we mentioned, our situation (and notation) is slightly
different, therefore we provide some details.

Let $D\subset\Omega$ with $|D|=\delta$, and $u$ be the
principal normalized eigenfunction associated to $\mu(D,\Omega)$. Then
$\Ccal_\alpha u \in H^{1}_{0}(\Ccal_\alpha(D),\Sigma_\alpha)$, and
$1=\|u\|_{L^{2}}=\|\Ccal_\alpha u\|_{L^{2}}$. Moreover, $ u $ and
$\Ccal_\alpha  u $ have the same distribution function
\begin{equation}\label{eq:defdistr}
f(t)=\left|\{x\in D:  u (x)>t\}\right|= |\{ x\in\Ccal_\alpha(D):
\Ccal_\alpha{ u }(x)>t\}|.
\end{equation}
Then Lemma \ref{lem:Kincone} implies
\begin{equation}\label{eq:isousata}
P(\{ u >t\},\Omega) = \Hcal^{N-1}(\{ u =t\} ) \geq \Hcal^{N-1}(\{\Ccal_\alpha  u =t\})
=P(\{\Ccal_\alpha u >t\},\Omega).
\end{equation}

On the other hand,   the co-area formula yields the following expressions concerning
the distribution function $f$
\begin{equation}\label{eq:fcoarea}
\begin{split}
f(t) &
=\int_{\{ u >t\}}1\,dx=\int_{t}^{+\infty}\left[
\int_{\{  u =s\} }|\nabla  u |^{-1}d\mathcal H^{N-1}\right]ds,
\\
|f'(t)|&
=\int_{\{ u =t\}} |\nabla  u |^{-1}d\mathcal H^{N-1},\qquad \mbox{for a.e. }t\in \R^+.
\end{split}
\end{equation}
Let us also observe that H\"older inequality implies
\[
\begin{split}
\mathcal H^{N-1}(\{  u =t\})^{2}&=\left(\int_{\{ u =t\}}d\mathcal H^{N-1 }\right)^{2}=\left[\int_{\{ u =t\}}|\nabla  u |
^{1/2}|\nabla  u |^{-1/2}d\mathcal H^{N-1 }\right]^{2}
\\
&\leq \left(\int_{\{ u =t\}}|\nabla  u |d\mathcal H^{N-1 }\right)
\left(\int_{\{ u =t\}}|\nabla  u |^{-1}d\mathcal H^{N-1 }\right)
= \left(\int_{\{ u =t\}}|\nabla  u |d\mathcal H^{N-1 }\right)
|f'(t)|.
\end{split}
\]
Exploiting this estimate, together with the co-area formula for $ u $, one has
\[
\begin{split}
\mu(D,\Omega)&=\int_{\Omega}|\nabla u|^{2}\,dx=\int_{0}^{+\infty}\left[\int_{\{ u =t\}}|\nabla  u |d\mathcal H^{N-1 }\right]dt
\\
&\geq\int_{0}^{+\infty}\left[\mathcal H^{N-1}(\{  u =t\})^{2}|f'(t)|^{-1}\right]dt
\geq
\int_{0}^{+\infty}\left[\mathcal H^{N-1}(\{\Ccal_\alpha{  u }=t\})^{2}|f'(t)|^{-1}\right]dt,
\end{split}
\]
where in the last passage we used the isoperimetric inequality \eqref{eq:isousata}. 
Taking into account \eqref{eq:defdistr}, that
$|\nabla \Ccal_\alpha{ u }|$ is constant on the level sets 
and applying the co-area formula again, we can carry on the above estimate writing
\[
\begin{split}
\mu(D,\Omega)&
\geq 
\int_{0}^{+\infty }
\left(\int_{\{\Ccal_\alpha{  u }=t\}} |\nabla \Ccal_\alpha{  u }|^{-1}d\mathcal
H^{N-1} \right)^{-1}\mathcal H^{N-1}(\{\Ccal_\alpha{  u }=t\})^{2}
dt
\\
&=
\int_{0}^{+\infty}\left[\int_{\{\Ccal_\alpha{  u }=t\}}|\nabla \Ccal_\alpha{  u }|d\mathcal H^{N-1}\right]dt
=\int_{\Sigma_\alpha}|\nabla \Ccal_\alpha{  u }|^{2}\,dx
\geq \mu(\Ccal_\alpha (D),\Sigma_\alpha),
\end{split}
\]
and the proposition follows.
\end{proof}
The above Proposition and  Lemma \ref{lem:Kincone} yield the following result.
\begin{corollary}\label{coro:cone_for_K}
Assume that, for some $\bar\delta$, there exists a cone $\Sigma_\alpha$, with
$\alpha = K^N(\Omega,\bar\delta)$, such that
\[
\Omega \cap B_{r(\alpha,\bar\delta)} = \Sigma_\alpha \cap B_{r(\alpha,\bar\delta)}.
\]
Then, for every $\delta\le\bar\delta$, we have $K(\Omega,\delta) = K(\Omega,\bar\delta)$,
\[
\sd(\delta) = K^2(\Omega,\delta)\lambda_1^\text{Dir} \cdot
\delta^{-2/N},
\]
and both $K(\Omega,\delta)$ and $\sd(\delta)$ are achieved by $D^*=B_{r(\alpha,\delta)}\cap\Omega$.
\end{corollary}
In order to complete the proof of Theorem  \ref{thm:main_polythope}, the last ingredient we miss is the explicit
evaluation of $K(\Omega,\delta)$ in case $\Omega$ is a planar rectangle, via the characterization of optimal sets. This is well-known in the
literature and we refer for example to~\cite{cianchi} for more details.
\begin{lemma}\label{lem:relisopineq}
Let $\Omega = (0,L_1) \times (0,L_2)$, with $L_1\le L_2$. Then
\begin{equation}\label{eq:rii}
K^2\left(\Omega,\frac{L_1^2}{\pi}\right) = \frac{\pi}{4}
\end{equation}
holds, and an optimal set is given by the quarter of disk centered at a vertex of
$\Omega$.
\end{lemma}
\begin{proof}
For $0<\delta<|\Om|$, let us define
\[
C(\Omega,\delta) := \inf{\left\{\frac14 \frac{P^2(D,\Om)}{\delta} : D\subset\Om,\ |D|= \delta\right\}}.
\]
Then
\[
K^2(\Omega,\bar\delta) = \inf_{0<\delta\le\bar\delta} C(\Omega,\delta),
\]
and we want to find the optimal set for $K^2$ in case $\bar\delta=L_1^2/\pi$.
In general, $K$ needs not to be achieved. On the other hand,
according to~\cite[Thm.~2 and Thm.~3]{cianchi} (see also
\cite[Thm.~4.6 and Thm.~5.12]{MR3335407}), $C(\Omega,\delta)$ is achieved by an
open, connected $E^*_\delta\subset\Omega$, such that $\partial E^*_\delta\cap \Om$ is
either an arc of circle or a straight line. Moreover, the Hausdorff measure of the
intersection of the boundaries satisfies $\mathcal H^1(\partial E^*_\delta\cap \partial
\Om)>0$, and $\partial E^*_\delta\cap \Om$ reaches the boundary of $\Om$ orthogonally
at flat points (i.e. not at a vertex). Finally, since $\delta\le\bar\delta<|\Om|/2$,
then $E^*_\delta$ is convex. Hence, there are four possible configurations for
$E^*_\delta$:
\begin{enumerate}
\item
$E^*_\delta$ is a half disk, centered at a flat point of $\partial\Om$;
\item
$E^*_\delta$ is a quarter of a disk $D_\delta$, centered at a vertex of $\partial\Om$;
\item
$E^*_\delta$ is a portion of a disk, with boundary either passing through two
vertices, or passing through one vertex and orthogonal to one side of
$\partial\Om$ (i.e., having endpoints on opposite sides of $\Om$);
\item
$E^*_\delta$ is (congruent to) the strip $S_\delta=(0,L_1)\times(0,\delta/L_1)$ (being $(0,\delta/L_2)\times(0,L_2)$ less convenient for $C$).
\end{enumerate}
It is easy to rule out configurations 1, which is always worse than 2
(because the perimeter of a half disk is bigger than that of a quarter of disk having
the same measure) and 3 in favor of 4 (because the perimeter of
such an $E^*_\delta$ is always bigger than $L_1$, and a strip with the same measure
and perimeter $L_1$ always exists). With respect to the alternatives 2 and 4, explicit computations show that
\[
\delta\le\frac{L_1^2}{\pi}
\qquad\implies\qquad
\frac14\frac{P^2(D_\delta,\Om)}{|D_\delta|}=\frac\pi4\le
\frac{L_1^2}{4\delta}=\frac14\frac{P^2(S_\delta,\Om)}{|S_\delta|},
\]
and the lemma follows.
\end{proof}
\begin{remark}\label{rem:bardelta}
Notice that finding explicit lower bounds of $\bar\delta$ for an $N$-dimensional orthotope,
even for $N=3$, is much more difficult: indeed in such case $C(\Om,\delta)$ is achieved
by a set having relative boundary with constant mean curvature, and therefore the cases 
to consider not only include planes, cylinders and spheres, but also other candidates
such as the Lawson surfaces and the Schwarz ones (see the survey
\cite{MR2167260} for more details).

On the other hand, in case $\Omega=(0,L_1)\times(0,L_2)$, the above isoperimetric
estimate is sharp: if $L_1^2/\pi<\delta\le L_1L_2/2$, then $K(\delta,\Om)$ is achieved
by the strip $S_\delta$. However, the situation for $\sd(\delta)$ may be different:
for instance, if $L_1=L_2$, then $\mu(S_\delta,\Om)>\mu(D_\delta,\Om)$ up to
$\delta\leq 1/2$, as one can verify by direct calculation. On the other hand if $L_1<<
L_2$, then the strip becomes an optimal set  also for $\sd(\delta)$,
$\delta \approx 1/2$, as one can see with the same argument
of~\cite[Example~5.4]{buve}.
\end{remark}
\begin{proof}[Proof of Theorem \ref{thm:main_polythope}]
The proof of the first part of the theorem follows by Corollary \ref{coro:cone_for_K}
and by \cite[Thm. 6.8]{MR3335407}, which implies that, for small volumes, the isoperimetric
regions in a convex polytope $\Omega$ are geodesic balls centered at vertices with the smallest solid
angle (recall also Lemma \ref{lem:Kincone}). The second part of the theorem is a consequence of the first part, and of Theorem \ref{thm:equality}. Finally, the estimate of $\bar\delta$ in case $\Omega$ is a planar rectangle follows from Lemma~\ref{lem:relisopineq}.
\end{proof}
\begin{proof}[Proof of Theorem \ref{thm:main_small} -- estimate from below]
In order to prove this estimate, we will combine Proposition \ref{prop:isoper_generale} with the asymptotic expansion of the relative isoperimetric profile obtained by Fall in~\cite{fall},
in the setting of Riemannian manifolds. More precisely, for $v>0$, the \emph{isoperimetric profile} relative to $\Om$ is the mapping
\[
v\mapsto I_\Om(v):=\min{\left\{P(D,\Om):D\subset\Om,\ |D|=v\right\}},
\]
and we define
\[
\hat{H}:=\max_{p\in \partial \Om}H(p) ,\qquad \beta_{N-1}:=\frac{N-1}{N(N+1)}\left(\frac{2}{\omega_N}\right)^{(N+1)/N}\omega_{N-1}.
\]
Then \cite[Corollary~1.3]{fall} yields
\begin{equation}\label{eq:fall}
I_\Om(v)= I_{\R^N_+}(v)\left(1-\beta_{N-1}\hat{H}\,v^{1/N}+O(v^{2/N})\right),\qquad \text{as }v\rightarrow 0.
\end{equation}
Since the half ball of volume $v$ has radius $r(v)=(2v/\omega_N)^{1/N}$, we infer that
\[
I_{\R^N_+}(v)=P(B^+_{r(v)},\R^N_+)=\frac{N\omega_N}{2}\left(\frac{2v}{\omega_N}\right)^{(N-1)/N}
=N\left(\frac{\omega_N}2\right)^{1/N}v^{(N-1)/{N}}.
\]
As a consequence
\[
K(\Om,\delta)=\inf_{0<v\leq \delta}\frac{I_\Om(v)}{Nv^{(N-1)/N}}=\left(\frac{\omega_N}2\right)^{1/N}\inf_{0<v\leq \delta}\left(1-\beta_{N-1}\hat H\,v^{1/N}+O(v^{2/N})\right),
\]
and finally
\[
K(\Om,\delta)=\left(\frac{\omega_N}2\right)^{1/N}\left(1-\beta_{N-1}\hat H\,\delta^{1/N}+o(\delta^{1/N})\right),\qquad\text{as }\delta\to0.
\]
Then Proposition \ref{prop:isoper_generale} applies, providing
\[
\inf_{|D|=\delta}\mu(D,\Om)\ge  \mu(B^+_1,\R^N_+)|B_1^+| ^{2/N}  \cdot \delta^{-2/N} \left( 1 -  \underline{C}_N \hat H\,\delta^{1/N} + o(\delta^{1/N})\right)\qquad \text{as }\delta\to0^+,
\]
where
\[
\underline{C}_N = 2\beta_{N-1} = 4 \frac{|B_1^+|^{-1/N}}{N} \frac{N-1}{N+1} \frac{\omega_{N-1}}{\omega_N}.\qedhere
\]
\end{proof}

\subsection{Spectral drops in regular domains -- asymptotic spherical shapes}

Recall that $\varphi\in H^1_0(B_1)$ denotes the first eigenfunction
of the Dirichlet Laplacian on $B_1$, with eigenvalue $\lambda_1^\text{Dir}$, see equation \eqref{eq:half_Rayleigh}.
We will show the following.
\begin{proposition}\label{prop:asymptotic_in_r}
Let $x_0\in\partial\Omega$ be such that $\partial\Omega$ is of class $C^2$ near $x_0$. For $r>0$ small, let $D_r = B_r(x_0) \cap \Omega$. Then
\[
\begin{split}
|D_r| 
&= |B_1^+| \, r^N \cdot \left( 1 -  \frac{N-1}{N+1} \frac{\omega_{N-1}}{\omega_N}
H(x_0)\, r + o(r)\right),\smallskip\\
\mu(D_r,\Omega) 
&\le \mu(B^+_1,\R^N_+)\, r^{-2} \cdot \left( 1 -
\frac{N-1}{4}\frac{\int_{|x'|<1}\varphi^2(x',0)\,dx'}{\int_{B_1^+}
|\nabla\varphi|^2\,dx}  H(x_0)\, r + o(r)\right),
\end{split}
\]
as $r\to0^+$, where $H(x_0)$ denotes the mean curvature of $\partial\Omega$ at $x_0$.
\end{proposition}
To prove the proposition, w.l.o.g. we choose $x_0=0$ and
$\psi\in C^2(B_1\cap\{x_N=0\})$, with $\psi(0) = 0$,
$\nabla \psi(0) = 0$, in such a way that $\Omega$ is (locally) the epigraph of $\psi$.
Then, for $r$ sufficiently small,
\[
D_r = B_r(0) \cap \Omega =B_r(0) \cap\left\{x:x_N>\psi(x')\right\} .
\]
We need some preliminary lemmas.
\begin{lemma}\label{lem:der_H_nosymm}
Let $f\in C(\overline{B}_1)$ and
$\psi$ as above. For $r\in(0,1)$ let
\[
h(r) := \int_{B_1 \cap \{x_N>\psi(rx')/r\}} f(x',x_N)\,dx.
\]
Then
\[
h'(0^+) = - \frac12 \int_{\{|x'|<1\}} f(x',0) D^2\psi(0)x'\cdot x'\,dx'.
\]
\end{lemma}
\begin{proof}
To start with, we extend $f$ to $\{x_N<-\sqrt{1-|x'|^2}\}$ by setting
\[
f(x',x_N) = f(x', -\sqrt{1-|x'|^2})\qquad\text{whenever }x_N \le -\sqrt{1-|x'|^2}.
\]
Then $f$ is continuous and bounded in $\{x:|x'|<1,\,x_N \le \sqrt{1-|x'|^2}\}$.
Moreover
\[
h(r) = \int_{\{|x'|<1\}}dx' \int_{\psi(rx')/r}^{\sqrt{1-|x'|^2}} f\,dx_N -
\int_{\{\psi(rx')/r<-\sqrt{1-|x'|^2}\}}dx' \int_{\psi(rx')/r}^{-\sqrt{1-|x'|^2}}
f\,dx_N = I_1(r) - I_2(r).
\]
We first show that $I_2(r) = o(r)$ as $r\to0^+$. Indeed, by assumption there exists
$\kappa\ge 0$ such that $\psi(x') \ge -\kappa |x'|^2$. Thus
\[
\Big\{x':\psi(rx')/r<-\sqrt{1-|x'|^2}\Big\} \subset
\Big\{x':-\kappa r |x'|^2<-\sqrt{1-|x'|^2}\Big\} \subset
\Big\{x':1 -\kappa^2 r^2 < |x'|^2 < 1\Big\}
\]
and
\[
|I_2(r)| \le \|f\|_\infty \int_{\{1 -\kappa^2 r^2 < |x'|^2 < 1\}}dx'
\int_{-\kappa r}^{0} \,dx_N = \|f\|_\infty \kappa r \int_{\{1 -\kappa^2 r^2 < |x'|^2 < 1\}}dx' = o(r).
\]
As a consequence
\[
h'(0^+) = I_1'(0^+) = \lim_{r\to0^+} \int_{\{|x'|<1\}}-f\left(x',\frac{\psi(rx')}{r}\right) \,
\frac{\nabla \psi(rx')\cdot rx' - \psi(rx')}{r^2} \,dx',
\]
and the lemma follows, as $\psi(rx') = r^2 D^2\psi(0)x'\cdot x'/2 + o(r^2)$,
as $r\to0^+$, uniformly in $|x'|\le1$.
\end{proof}
The mean curvature of the graph of $\psi$ at $0$ appears in the above estimate, in case $f$ is
symmetric.
\begin{lemma}\label{lem:der_H_symm}
Under the assumptions of Lemma \ref{lem:der_H_nosymm}, assume furthermore that
$f(x',0)$ is radially symmetric in $x'$. Then
\[
h'(0^+) = - \frac{H(0)}{2} \int_{\{|x'|<1\}} f(x',0) |x'|^2\,dx',
\]
where $H(0)=\trace (D^2\psi(0))/(N-1)$ is the mean curvature of the graph of
$\psi$ at $0$.
\end{lemma}
\begin{proof}
Up to a rotation we have
\[
\psi(x') = \frac12 \sum_{i=1}^{N-1}\kappa_i x_i^2 + o (|x'|^2)
\]
as $x'\to0$. On the other hand, since $f$ is symmetric,
\[
\int_{\{|x'|<1\}} f(x',0) x_i^2\,dx' = \int_{\{|x'|<1\}} f(x',0) x_j^2\,dx'
= \frac{1}{N-1} \int_{\{|x'|<1\}} f(x',0) |x'|^2\,dx',
\]
and Lemma \ref{lem:der_H_nosymm} yields
\[
h'(0^+) = - \frac12 \int_{\{|x'|<1\}} f(x',0) \sum_{i=1}^{N-1}\kappa_i x_i^2\,dx'
=- \frac{\sum_{i=1}^{N-1}\kappa_i}{2(N-1)} \int_{\{|x'|<1\}} f(x',0)  |x'|^2\,dx'.
\qedhere
\]
\end{proof}
Using the above result we can readily estimate $|D_r|$.
\begin{lemma}\label{lem:misura_di_Dr}
Under the above notations,
\[
|D_r| = \frac12 \omega_N \, r^N - \frac12 \frac{N-1}{N+1} \omega_{N-1} H(0)\, r^{N+1}
+ o(r^{N+1}).
\]
\end{lemma}
\begin{proof}
Writing $y=rx$ we have
\[
r^{-N} |D_r| = r^{-N} \int_{B_r \cap \Omega} \,dy =  \int_{B_1 \cap \{x_N>\psi(rx')/r\}} \,dx = h(r),
\]
where $h$ is defined as in the above lemmas, with $f\equiv 1$. Then
\[
h(0^+) = \int_{B_1 \cap \{x_N>0\}} \,dx = \frac12 \omega_N
\]
and, by Lemma \ref{lem:der_H_symm},
\[
h'(0^+) = - \frac{H(0)}{2} \int_{\{|x'|<1\}} |x'|^2\,dx' = - \frac{H(0)}{2}
\int_0^1 t^N
(N-1)\omega_{N-1} \,dt =- \frac12 \frac{N-1}{N+1} \omega_{N-1} H(0). \qedhere
\]
\end{proof}
We will estimate $\mu(D_r,\Omega)$ with the Rayleigh quotient of a rescaling
of the Dirichlet eigenfunction $\varphi$. Let $\tilde \varphi (y) := \varphi (y/r) \in H^1_0 (B_r)$. Then
\begin{equation}\label{eq:half_Rayleigh2}
\mu(D_r,\Omega)\le \frac{\int_{D_r}|\nabla\tilde\varphi(y)|^2\,
dy}{\int_{D_r}\tilde\varphi^2(y)\,dy} = \frac{r^{-2}
\int_{B_1 \cap \{x_N>\psi(rx')/r\}}
|\nabla\varphi(x)|^2\,dx}{\int_{B_1 \cap \{x_N>\psi(rx')/r\}}\varphi^2(x)
\,dx}=:r^{-2}R(r).
\end{equation}
\begin{proof}[End of the proof of Proposition \ref{prop:asymptotic_in_r}]
By \eqref{eq:half_Rayleigh} we have that $R(0^+) = \mu(B^+_1,\R^N_+)$.
In view of Lemma \ref{lem:misura_di_Dr} and equation \eqref{eq:half_Rayleigh2}
we only need to evaluate $R'(0^+)$. Using repeatedly Lemma
\ref{lem:der_H_symm} we obtain
\[
\begin{split}
R'(0^+) &=  - \frac{ H(0) }{2} \frac{
\int_{B_1^+}\varphi^2 \cdot 
\int_{\{|x'|<1\}} |\nabla\varphi(x',0)|^2 |x'|^2
- \int_{B_1^+}|\nabla\varphi|^2 \cdot 
\int_{\{|x'|<1\}} \varphi^2(x',0) |x'|^2
}{
\left(
\int_{B_1^+}\varphi^2
\right)^2}\\
&=  - \frac{ H(0) }{2\int_{B_1^+}\varphi^2
} \cdot
\int_{\{|x'|<1\}} \left[|\nabla\varphi(x',0)|^2 -
\lambda_1^\text{Dir} \varphi^2(x',0)\right]|x'|^2\,dx'.
\end{split}
\]
Recalling that $\varphi$ is radial,
with some abuse of notation we write, for $\rho=|x|$, $\varphi(x) = \varphi(\rho)$
and $|\nabla \varphi(x)| = - \varphi_\rho(\rho)$. This yields
\[
\begin{split}
-\frac{2\int_{B_1^+}\varphi^2}{ H(0) } R'(0^+) &=
\int_0^1 \left[\varphi_\rho^2 -
\lambda_1^\text{Dir} \varphi^2\right]\rho^2\cdot (N-1)\omega_{N-1}\rho^{N-2}\,d\rho\\
&=
\int_0^1 \left[\rho^N\varphi_\rho^2 + \rho^N\varphi\left(\varphi_{\rho\rho} + \frac{N-1}{\rho}
\varphi_\rho\right)\right]\cdot (N-1)\omega_{N-1}\,d\rho
\\
&=
\int_0^1 \left[\left(\rho^N\varphi\varphi_\rho\right)'
-\frac12(\varphi^2)'\rho^{N-1} \right]\cdot (N-1)\omega_{N-1}\,d\rho,
\end{split}
\]
and finally, integrating by parts,
\[
-\frac{2\int_{B_1^+}\varphi^2}{ H(0) } R'(0^+) 
=   \frac{ N-1 }{2}
\int_0^1 \varphi^2 \rho^{N-2} \cdot (N-1)\omega_{N-1}\,d\rho
=  \frac{ N-1 }{2}\int_{\{|x'|<1\}}\varphi^2(x',0)\,dx'.\qedhere
\]
\end{proof}
The last ingredient we need to conclude the proof of Theorem \ref{thm:main_small} is the following elementary lemma.
\begin{lemma}\label{lem:elementary_asymptotics}
Assume that, for positive constants $a,b,c,d$,
\[
\delta = a r^N \left( 1 - br + o(r)\right),\qquad
\mu = c r^{-2} \left( 1 - dr + o(r)\right),\qquad \text{as }r\to0^+.
\]
Then
\[
\mu = c a^{2/N}  \delta^{-2/N} \left( 1 -  \frac{a^{-1/N} (2b + Nd)}{N} \,\delta^{1/N} + o(\delta^{1/N})\right)\qquad \text{as }\delta\to0^+.
\]
\end{lemma}
\begin{proof}
From the expansion of $\delta$ we have that
\[
\delta^{1/N} = a^{1/N} r \left( 1 - br + o(r)\right)^{1/N} = a^{1/N} r \left( 1 - \frac{b}{N}r + o(r)\right),
\]
which implies
\[
r = a^{-1/N} \delta^{1/N} \left( 1 + \frac{a^{-1/N} b}{N}\, \delta^{1/N} + o(\delta^{1/N})\right),
\]
and the lemma follows.
\end{proof}
\begin{proof}[Proof of Theorem \ref{thm:main_small} -- estimate from above]
In the notation of Proposition \ref{prop:asymptotic_in_r}, let $\delta>0$ be sufficiently small and let $r=r(\delta)$ be such that $|D_r|=\delta$. Then we can apply Lemma
\ref{lem:elementary_asymptotics} to write
\begin{equation}\label{eq:dasoprax0}
\mu(D_{r(\delta)},\Omega) \le \mu(B^+_1,\R^N_+)|B_1^+| ^{2/N}  \cdot \delta^{-2/N} \left( 1 -  \overline{C}_N H(x_0)\,\delta^{1/N} + o(\delta^{1/N})\right)\qquad \text{as }\delta\to0^+,
\end{equation}
where
\[
\overline{C}_N = \frac{|B_1^+|^{-1/N}}{N} \left(2\frac{N-1}{N+1} \frac{\omega_{N-1}}{\omega_N}
+ N \frac{N-1}{4}\frac{\int_{|x'|<1}\varphi^2(x',0)\,dx'}{\int_{B_1^+}
|\nabla\varphi|^2\,dx} \right).\qedhere
\]
\end{proof}
\begin{proof}[Proof of Corollary~\ref{cor:post_small}]
Recalling that $\la(\beta, B_{r(\delta)}(x_0))\leq \mu(B_{r(\delta)}(x_0),\Om)$ (see Lemma~\ref{le:stime}), Theorem~\ref{thm:equality} yields
\[
1\leq \frac{\la(\beta, B_{r(\delta)}(x_0))}{\od(\beta,\delta)}\leq \frac{\mu(B_{r(\delta)}(x_0),\Om)}{\sd(\delta+\eps)(1-\sqrt{\delta/(\eps\beta)})^2},
\]
for all $\eps\in (\delta/\beta, |\Om|-\delta)$. Theorem~\ref{thm:main_small} allows to estimate the above term, obtaining
\[
\sd(\delta+\eps)\ge \mu(B^+_1,\R^N_+)|B_1^+| ^{2/N} \left[ (\delta+\eps)^{-2/N} -\underline{C}_N\hat H\cdot(\delta+\eps)^{-1/N}+o((\delta+\eps)^{-1/N})\right],
\]
\[
\mu(B_{r(\delta)}(x_0),\Om)\le \mu(B^+_1,\R^N_+)|B_1^+| ^{2/N} \left[   \delta^{-2/N} -\overline{C}_N H(x_0)\cdot\delta^{-1/N}+o(\delta^{-1/N})\right]
\]
as $\eps,\delta\to0$ (see \eqref{eq:dasoprax0}). Plugging this information into the above ratio we obtain
\[
\frac{\la(\beta, B_{r(\delta)}(x_0))}{\od(\beta,\delta)}\leq \frac{\delta^{-2/N}}{(\delta+\eps)^{-2/N}}\Big(1-\sqrt{\delta/(\eps\beta)}\Big)^{-2}\frac{1-\overline{C}_NH(x_0)\delta^{1/N}+o(\delta^{1/N})}{1-\underline{C}_N\hat H(\delta+\eps)^{1/N}+o((\delta+\eps)^{1/N})}.
\]
Finally, letting $\beta\to+\infty$ and choosing $\eps=\delta\beta^{-1/3}>\delta/\beta$, we have
\[
\frac{\la(\beta, B_{r(\delta)}(x_0))}{\od(\beta,\delta)}\leq\Big(1+\frac{2}{N}\beta^{-1/3}+o(\beta^{-1/3})\Big)\Big(1+2\beta^{-1/3}+o(\beta^{-1/3})\Big)\frac{1-\overline{C}_N
H(x_0)\delta^{1/N}+o(\delta^{1/N})}{1-\underline{C}_N\hat H\delta^{1/N}+o(\delta^{1/N})},
\]
which proves the claim with $A = \underline{C}_N\hat H - \overline{C}_N H(x_0)$ and $B = \frac2N + 2$.
\end{proof}

\bibliography{frac_eig}

\begin{thebibliography}{10}

\bibitem{alca}
H.~W. Alt and L.~A. Caffarelli.
\newblock Existence and regularity for a minimum problem with free boundary.
\newblock {\em J. Reine Angew. Math.}, 325:105--144, 1981.

\bibitem{MR572958}
C.~Bandle.
\newblock {\em Isoperimetric inequalities and applications}, volume~7 of {\em
  Monographs and Studies in Mathematics}.
\newblock Pitman (Advanced Publishing Program), Boston, Mass.-London, 1980.

\bibitem{MR2214420}
H.~Berestycki, F.~Hamel, and L.~Roques.
\newblock Analysis of the periodically fragmented environment model. {I}.
  {S}pecies persistence.
\newblock {\em J. Math. Biol.}, 51(1):75--113, 2005.

\bibitem{bhp}
T.~Brian\c{c}on, M.~Hayouni, and M.~Pierre.
\newblock Lipschitz continuity of state functions in some optimal shaping.
\newblock {\em Calc. Var. Partial Differential Equations}, 23(1):13--32, 2005.

\bibitem{brla}
T.~Brian\c{c}on and J.~Lamboley.
\newblock Regularity of the optimal shape for the first eigenvalue of the
  {L}aplacian with volume and inclusion constraints.
\newblock {\em Ann. Inst. H. Poincar{\'e} Anal. Non Lin{\'e}aire},
  26(4):1149--1163, 2009.

\bibitem{bubu}
D.~Bucur and G.~Buttazzo.
\newblock {\em Variational methods in shape optimization problems}, volume~65
  of {\em Progress in Nonlinear Differential Equations and their Applications}.
\newblock Birkh\"{a}user Boston, Inc., Boston, MA, 2005.

\bibitem{bmpv}
D.~Bucur, D.~Mazzoleni, A.~Pratelli, and B.~Velichkov.
\newblock Lipschitz regularity of the eigenfunctions on optimal domains.
\newblock {\em Arch. Ration. Mech. Anal.}, 216(1):117--151, 2015.

\bibitem{buve}
G.~Buttazzo and B.~Velichkov.
\newblock The spectral drop problem.
\newblock In {\em Recent advances in partial differential equations and
  applications}, volume 666 of {\em Contemp. Math.}, pages 111--135. Amer.
  Math. Soc., Providence, RI, 2016.

\bibitem{MR1014659}
R.~S. Cantrell and C.~Cosner.
\newblock Diffusive logistic equations with indefinite weights: population
  models in disrupted environments.
\newblock {\em Proc. Roy. Soc. Edinburgh Sect. A}, 112(3-4):293--318, 1989.

\bibitem{MR1105497}
R.~S. Cantrell and C.~Cosner.
\newblock The effects of spatial heterogeneity in population dynamics.
\newblock {\em J. Math. Biol.}, 29(4):315--338, 1991.

\bibitem{MR2191264}
R.~S. Cantrell and C.~Cosner.
\newblock {\em Spatial ecology via reaction-diffusion equations}.
\newblock Wiley Series in Mathematical and Computational Biology. John Wiley \&
  Sons, Ltd., Chichester, 2003.

\bibitem{MR1796024}
S.~Chanillo, D.~Grieser, M.~Imai, and K.~a. Kurata.
\newblock Symmetry breaking and other phenomena in the optimization of
  eigenvalues for composite membranes.
\newblock {\em Comm. Math. Phys.}, (2):315--337, 2000.

\bibitem{MR2421158}
S.~Chanillo and C.~E. Kenig.
\newblock Weak uniqueness and partial regularity for the composite membrane
  problem.
\newblock {\em J. Eur. Math. Soc. (JEMS)}, 10(3):705--737, 2008.

\bibitem{MR2473259}
S.~Chanillo, C.~E. Kenig, and T.~To.
\newblock Regularity of the minimizers in the composite membrane problem in
  {$\mathbb{R}^2$}.
\newblock {\em J. Funct. Anal.}, 255(9):2299--2320, 2008.

\bibitem{cianchi}
A.~Cianchi.
\newblock On relative isoperimetric inequalities in the plane.
\newblock {\em Boll. Un. Mat. Ital. B (7)}, 3(2):289--325, 1989.

\bibitem{dl}
M.~Dambrine and J.~Lamboley.
\newblock Stability in shape optimization with second variation.
\newblock {\em J. Differential Equations}, 267(5):3009--3045, 2019.

\bibitem{deve}
G.~De~Philippis and B.~Velichkov.
\newblock Existence and regularity of minimizers for some spectral functionals
  with perimeter constraint.
\newblock {\em Appl. Math. Optim.}, 69(2):199--231, 2014.

\bibitem{desilva}
D.~De~Silva.
\newblock Free boundary regularity for a problem with right hand side.
\newblock {\em Interfaces Free Bound.}, 13(2):223--238, 2011.

\bibitem{dego}
A.~Derlet, J.-P. Gossez, and P.~Tak{{\'a}}{\v{c}}.
\newblock Minimization of eigenvalues for a quasilinear elliptic {N}eumann
  problem with indefinite weight.
\newblock {\em J. Math. Anal. Appl.}, 371(1):69--79, 2010.

\bibitem{MR1897460}
O.~Druet.
\newblock Sharp local isoperimetric inequalities involving the scalar
  curvature.
\newblock {\em Proc. Amer. Math. Soc.}, 130(8):2351--2361, 2002.

\bibitem{evansgariepy}
L.~C. Evans and R.~F. Gariepy.
\newblock {\em Measure theory and fine properties of functions}.
\newblock Studies in Advanced Mathematics. CRC Press, Boca Raton, FL, 1992.

\bibitem{fall}
M.~M. Fall.
\newblock Area-minimizing regions with small volume in {R}iemannian manifolds
  with boundary.
\newblock {\em Pacific J. Math.}, 244(2):235--260, 2010.

\bibitem{fisher}
R.~Fisher.
\newblock The advance of advantageous genes.
\newblock {\em Ann. Eugenics}, 7:335--369, 1937.

\bibitem{henrot}
A.~Henrot.
\newblock {\em Extremum problems for eigenvalues of elliptic operators}.
\newblock Frontiers in Mathematics. Birkh\"auser Verlag, Basel, 2006.

\bibitem{ho}
A.~Henrot and E.~Oudet.
\newblock Minimizing the second eigenvalue of the {L}aplace operator with
  {D}irichlet boundary conditions.
\newblock {\em Arch. Ration. Mech. Anal.}, 169(1):73--87, 2003.

\bibitem{hp}
A.~Henrot and M.~Pierre.
\newblock {\em Variation et optimisation de formes}, volume~48 of {\em
  Math\'ematiques \& Applications (Berlin) [Mathematics \& Applications]}.
\newblock Springer, Berlin, 2005.
\newblock Une analyse g\'eom\'etrique. [A geometric analysis].

\bibitem{jesa}
D.~Jerison and O.~Savin.
\newblock Some remarks on stability of cones for the one-phase free boundary
  problem.
\newblock {\em Geom. Funct. Anal.}, 25(4):1240--1257, 2015.

\bibitem{MR2494032}
C.-Y. Kao, Y.~Lou, and E.~Yanagida.
\newblock Principal eigenvalue for an elliptic problem with indefinite weight
  on cylindrical domains.
\newblock {\em Math. Biosci. Eng.}, 5(2):315--335, 2008.

\bibitem{derek}
D.~Kielty.
\newblock Singular limits of sign-changing weighted eigenproblems.
\newblock {\em ArXiv e-prints}, arxiv:1812.03617, 2018.

\bibitem{kpp}
A.~Kolmogorov, I.~Petrovsky, and N.~Piskunov.
\newblock \'{E}tude del'\'equation de la diffusion avec croissance de la
  quantit\'e de mati\`ere et son application \`a un probl\`eme biologique.
\newblock {\em Bulletin Universit\'e d'\'Etat \`a Moscou (Bjul. Moskowskogo
  Gos. Univ.), S\'erie internationale A}, 1:1--26, 1937.

\bibitem{llnp}
J.~Lamboley, A.~Laurain, G.~Nadin, and Y.~Privat.
\newblock Properties of optimizers of the principal eigenvalue with indefinite
  weight and {R}obin conditions.
\newblock {\em Calc. Var. Partial Differential Equations}, 55(6):Paper No. 144,
  37, 2016.

\bibitem{yyli}
Y.~Y. Li.
\newblock On a singularly perturbed equation with {N}eumann boundary condition.
\newblock {\em Comm. Partial Differential Equations}, 23(3-4):487--545, 1998.

\bibitem{MR1000160}
P.-L. Lions and F.~Pacella.
\newblock Isoperimetric inequalities for convex cones.
\newblock {\em Proc. Amer. Math. Soc.}, 109(2):477--485, 1990.

\bibitem{MR963504}
P.-L. Lions, F.~Pacella, and M.~Tricarico.
\newblock Best constants in {S}obolev inequalities for functions vanishing on
  some part of the boundary and related questions.
\newblock {\em Indiana Univ. Math. J.}, 37(2):301--324, 1988.

\bibitem{ly}
Y.~Lou and E.~Yanagida.
\newblock Minimization of the principal eigenvalue for an elliptic boundary
  value problem with indefinite weight, and applications to population
  dynamics.
\newblock {\em Japan J. Indust. Appl. Math.}, 23(3):275--292, 2006.

\bibitem{mapeve_matrix}
D.~Mazzoleni, B.~Pellacci, and G.~Verzini.
\newblock Quantitative analysis of a singularly perturbed shape optimization
  problem in a polygon.
\newblock In {\em 2018 {MATRIX} annals}, MATRIX Book Ser. Springer, Cham, to
  appear.

\bibitem{mtv}
D.~Mazzoleni, S.~Terracini, and B.~Velichkov.
\newblock Regularity of the optimal sets for some spectral functionals.
\newblock {\em Geom. Funct. Anal.}, 27(2):373--426, 2017.

\bibitem{MoPeVe}
E.~Montefusco, B.~Pellacci, and G.~Verzini.
\newblock Fractional diffusion with {N}eumann boundary conditions: the logistic
  equation.
\newblock {\em Discrete Contin. Dyn. Syst. Ser. B}, 18(8):2175--2202, 2013.

\bibitem{MR876139}
F.~Pacella and M.~Tricarico.
\newblock Symmetrization for a class of elliptic equations with mixed boundary
  conditions.
\newblock {\em Atti Sem. Mat. Fis. Univ. Modena}, 34(1):75--93, 1985/86.

\bibitem{MR3771424}
B.~Pellacci and G.~Verzini.
\newblock Best dispersal strategies in spatially heterogeneous environments:
  optimization of the principal eigenvalue for indefinite fractional {N}eumann
  problems.
\newblock {\em J. Math. Biol.}, 76(6):1357--1386, 2018.

\bibitem{MR3335407}
M.~Ritor\'{e} and E.~Vernadakis.
\newblock Isoperimetric inequalities in {E}uclidean convex bodies.
\newblock {\em Trans. Amer. Math. Soc.}, 367(7):4983--5014, 2015.

\bibitem{haro}
L.~Roques and F.~Hamel.
\newblock Mathematical analysis of the optimal habitat configurations for
  species persistence.
\newblock {\em Math. Biosci.}, 210(1):34--59, 2007.

\bibitem{MR2167260}
A.~Ros.
\newblock The isoperimetric problem.
\newblock In {\em Global theory of minimal surfaces}, volume~2 of {\em Clay
  Math. Proc.}, pages 175--209. Amer. Math. Soc., Providence, RI, 2005.

\bibitem{MR0043440}
J.~G. Skellam.
\newblock Random dispersal in theoretical populations.
\newblock {\em Biometrika}, 38:196--218, 1951.

\end{thebibliography}
\bibliographystyle{abbrv}

\end{document}